\documentclass[onefignum,onetabnum]{siamart190516}

\usepackage{lipsum,bm}
\usepackage{amsfonts,wasysym}
\usepackage{graphicx}
\usepackage{epstopdf}
\usepackage{algorithmic}
\usepackage{amsopn}
\usepackage{amssymb}
\usepackage{mathtools}
\usepackage{verbatim}
\usepackage{enumitem}
\ifpdf
\DeclareGraphicsExtensions{.eps,.pdf,.png,.jpg}
\else
\DeclareGraphicsExtensions{.eps}
\fi

\newcommand{\eps}{\varepsilon}

\newsiamremark{Remark}{Remark}
\newsiamremark{Definition}{Definition}
\newsiamthm{Corollary}{Corollary}
\newsiamthm{Assumption}{Assumption}
\newsiamthm{Theorem}{Theorem}
\newsiamthm{Proposition}{Proposition}
\newsiamthm{Lemma}{Lemma}

\newcommand{\E}{\mathbb{E}}
\newcommand{\Z}{\mathbb{Z}}
\newcommand{\R}{\mathbb{R}}
\newcommand{\N}{\mathbb{N}}

\newcommand{\var}[1]{\mathrm{Var}\left(#1\right)}

\newcommand{\x}{\boldsymbol{x}}
\newcommand{\y}{\boldsymbol{y}}
\newcommand{\uu}{\boldsymbol{u}}

\newcommand{\eexp}[1]{\exp\left\{#1\right\}}

\newcommand{\bvar}[1]{\mathrm{Var}\bigg(#1\bigg)}

\headers{QMC for unbounded integrands with IS}{D. Ouyang, X. Wang, and Z. He}

\title{Quasi-Monte Carlo for unbounded integrands with importance sampling\thanks{Submitted to the editors DATE.
		\funding{This work of the second author was funded by the National Science Foundation of China (No. 720711119). And the third author was funded by the National Science Foundation of China (No. 12071154), Guangdong Basic and Applied Basic Research Foundation (No. 2021A1515010275). }}
  }

\author{Du Ouyang\thanks{Department of Mathematical Sciences, Tsinghua University, Beijing 100084, People's Republic of China (\email{oyd21@mails.tsinghua.edu.cn}).}\and
	 Xiaoqun Wang\thanks{Department of Mathematical Sciences, Tsinghua University, Beijing 100084, People's Republic of China (\email{wangxiaoqun@mail.tsinghua.edu.cn}).} \and Zhijian He\thanks{Corresponding author. School of Mathematics, South China University of Technology, Guangzhou 510641, People's Repulic of China (\email{hezhijian@scut.edu.cn}).}}

\ifpdf
\hypersetup{
	pdftitle={Quasi-Monte Carlo for unbounded integrands with importance sampling},
	pdfauthor={Du Ouyang\and Xiaoqun Wang \and Zhijian He}
}
\fi

\begin{document}
	
	\maketitle
	
	\begin{abstract}
		We consider the problem of estimating an expectation $ \E\left[ h(W)\right]$ by quasi-Monte Carlo (QMC) methods, where $ h $ is an unbounded smooth function on $ \R^d $ and $ W$ is a standard normal distributed random variable. To study rates of convergence for QMC on unbounded integrands, we use a smoothed projection operator  to project the output of $W$ to a bounded region, which differs from the strategy of avoiding the singularities along the boundary of the unit cube $ [0,1]^d $ in \cite{owen2006a}.
        The error is then bounded by the quadrature error of the transformed integrand and the projection error. If the function $h(\x)$ and its mixed partial derivatives do not grow too fast as the Euclidean norm $\lvert\bm x\rvert$ goes to infinity, we obtain an error rate of $O(n^{-1+\eps})$ for QMC and randomized QMC (RQMC) with a sample size $n$ and an arbitrarily small $\eps>0$. However, the rate turns out to be $O(n^{-1+2M+\eps})$ if the functions grow exponentially with a rate of $O(\exp\{M\lvert\bm x\rvert^2\})$ for a constant $M\in(0,1/2)$. Superisingly, we find that using importance sampling with t distribution as the proposal can improve the root mean squared error of RQMC from $O(n^{-1+2M+\eps})$ to $O( n^{-3/2+\eps})$ for any $M\in(0,1/2)$.
	\end{abstract}
	
	
	
\section{Introduction}\label{sec:intro}

Quasi-Monte Calro (QMC) is an efficient quadrature method to numerically solve the integral problems on the unit cube $ \left[ 0,1\right]^d$. Unlike the Monte Carlo (MC) method, the QMC method uses low-discrepancy sequences instead of  random sequences (see~\cite{caflisch1998,Niederreiter1992,caflisch1997,dick2010,owen2013}). If the integrand has bounded
variation in the sense of Hardy and Krause (BVHK), then by using the Koksma-Hlawka inequality, QMC with $n$ quadrature points yields a deterministic error of $O(n^{-1}(\log n)^{d})$, which is asymptotically faster than the MC rate $O(n^{-1/2})$.

In many problems of financial engineering and stochastic control, the underlying solutions can be formulated as expectations of the form $\E[h(W)]$ with respect to a standard normal distribution $W$ (see~\cite{hull2003,glasserman2004,zhang2020}). To estimate $\E[h(W)]$, QMC quadrature rule takes
\begin{equation}\label{eq:qmcapproximate}
 \widehat I_n(h) = \frac{1}{n}\sum_{j = 1}^{n} h\circ \Phi^{-1}(\boldsymbol{y}_j),
\end{equation}
where $\{\boldsymbol{y}_j\}_{j=1}^n$ is a low-discrepancy point set in $[0,1]^{d}$, $\Phi(x)$ is the cumulative distribution function (CDF) of $N(0,1)$ satisfying
\begin{equation}\label{eq:distribute}
    \Phi(x) = \int_{-\infty}^x \frac{1}{\sqrt{2\pi}} e^{-u^2/2} du,
\end{equation}
$\circ$ is the composite operator and $ \Phi^{-1}(\y)$ is the inverse of $\Phi$ acting on each component of the argument $\y$. 
Since unbounded functions cannot be BVHK, the Koksma-Hlawka inequality fails to provide the $O(n^{-1}(\log n)^{d})$ error rate for unbounded functions. This paper is devoted to providing comprehensive error analysis for \eqref{eq:qmcapproximate}, in which the integrand $h\circ \Phi^{-1}$ may have singularities along the boundary of the unit cube $[0,1]^d$.

Owen~\cite{owen2006a} studied the QMC error for unbounded functions on $ \left( 0,1\right)^d$. He found that QMC error attains a rate of $O(n^{-1+\max_j A_j + \eps})$ if the integrand satisfies the boundary growth condition
\begin{equation}\notag
    \left| \partial^{\uu}f(\y)\right| \le B\prod_{j =1}^d\min\left( y_j,1-y_j\right)^{-A_j-\mathbf{1}_{j\in \uu}}
\end{equation}
for some $ A_j >0$, $ 0< B< \infty$, and all $ \uu\subseteq 1{:}d= \left\{ 1,\dots,d\right\}$, where $ \partial^{\uu}f(\y)$ denotes the mixed partial  derivative of $f(\y)$ with respect to $y_j$ with $j\in \uu$ and $\eps>0$ is arbitrarily small.
Moreover, randomized QMC (RQMC) method yields the same rate of $ O(n^{-1+\max_j A_j + \eps})$ for the mean error. Recently, He et al.~\cite{he2023} found that this rate also holds for root mean squared error (RMSE) when using scrambled nets (a commonly used RQMC method \cite{owen1995}). It is easy to see that when $A_j$ are all arbitrarily small (we call it the ``QMC-friendly" growth condition), the convergence rate achieves the optimal case $O(n^{-1+ \eps})$. In this article, we propose a projection based quasi-Monte Carlo (P-QMC) method and further refine the boundary growth conditions considered in Owen~\cite{owen2006a}. Using P-QMC or RQMC method, we obtain better convergence results for different boundary growth conditions. 

There are some related work on studying unbounded integrands in the context of QMC. Kuo et al.~\cite{kuo2006a} studied the problem of multivariate integration over $ \R^d$. They considered the case where the intergrand belongs to some weighted tensor product reproducing kernel Hilbert space and proved that good randomly shifted lattice rules can be constructed component by component to achieve a worst case error of order $ O(n^{-1/2})$. Moreover, Kuo et al.~\cite{kuo2010} improved the results by proving that a rate of convergence close to the optimal order $ O(n^{-1})$ can be achieved with an appropriate choice of parameters for the function
space. Based on this, Nichols and Kuo~\cite{kuo2014} extended the theory of Kuo et al.~\cite{kuo2010} in several non-trivial directions. 
In the above work, the constants in the big-$ O$ bounds can be independent
of dimension $ d$ under appropriate conditions on the weights of the function space. Nuyens and Suzuki~\cite{nuyens2023} introduced a method that scales lattice rules from the unit cube $ [0,1]^d$ to properly sized
boxes on $ \R^d$ and achieved higher-order convergence that matches the
smoothness of the integrand in a certain Sobolev space. Basu and Owen~\cite{basu2016b}
 studied three quadrature methods for integrands on the square
$ [0,1]^2$ that may become singular as the point approaches the diagonal line $ x_1 = x_2$.

Importance sampling (IS) is an efficient variance reduction method in the context of MC (see~\cite{glasserman1999b,kuk1999,owen2000,dick2019,zhang2021}). However, IS cannot improve the convergence rate of MC. It is natural to ask a question: ``Can IS accelerate the convergence rate in QMC?". He et al.~\cite{he2023} followed the framework of Owen~\cite{owen2006a} to show that using a proper IS in RQMC can retain the RMSE rate of $O(n^{-1+ \eps})$ under the ``QMC-friendly" growth condition. In our framework, by employing suitable IS to slow down the growth of the integrand, the error rate of RQMC is improved from $O(n^{-1+\max_j A_j+\eps})$ to $ O(n^{-3/2+\eps})$, in which the growth condition is not necessarily ``QMC-friendly". 

A key strategy in \cite{owen2006a} is to employ an auxiliary function that has a low variation to approximate the unbounded integrand. However, the auxiliary function used in \cite{owen2006a} is not smooth enough so that it does not meet the smoothness requirement in Owen~\cite{owen2008} for establishing the $ O(n^{-3/2+\eps})$ error rate. Differently, the function constructed based on the projection method is smooth, allowing us to address this issue and obtain the desired rate of convergence when using IS.

In this paper, we prove that when the smooth integrand grows at rate of \\$ O(\eexp{M|\x|^k})$ with $ M >0 $ and $ 0<k<2$, the convergence rates of the P-QMC and RQMC methods is $ O(n^{-1+\eps})$. By Owen~\cite{owen2006a}, the convergence rate of RQMC turns out to be $ O(n^{-1+2M+\eps})$ if the integrand grows extremely fast with a rate of $ O(\eexp{M|\x|^2})$ for a constant $ 0<M<1/2$. However, we demonstrate that with an  appropriate IS, the convergence rate of RQMC improves from $O(n^{-1+2M+\eps}) $ to $ O(n^{-3/2+\eps})$. Our work theoretically establishes that suitable IS can accelerate the convergence rate of QMC. Some integrands after pre-integration (also known as conditioning) \cite{griewank2018high,zhang2020} arsing from option pricing and smooth loss functions in deep learning both satisfy our growth conditions. Leveraging the findings in this paper, we can obtain the convergence rates of the QMC methods for these problems.

	 
The structure of this paper is as follows. Section~\ref{sec:pre} introduces the basic concepts of QMC and RQMC methods. Section~\ref{sec:growth} focuses on the function growth conditions as described in Owen~\cite{owen2006a}, and categorizes these into specific function growth classes. In Section~\ref{sec:pqmc}, we propose the P-QMC method and obtain the convergence results with respect to different growth classes. Section~\ref{sec:is} discusses the P-QMC and RQMC methods using importance sampling, demonstrating that the convergence rates can be significantly improved by using a proper IS. Section~\ref{sec:numerical} provides numerical experiments to confirm the theoretical results. Section~\ref{sec:conclusion} concludes the paper. We put lengthy but useful results about the operations of growth classes in Appendix.

\section{Preliminaries}\label{sec:pre} 

In this article, all norms that appear are Euclidean norms. In order to avoid ambiguity, bold symbols, such as $ \x$ and $ \y$, are used to represent vectors, and normal symbols are used to represent scalars. For example, we use $ \y_j$ to represent an element of a low-discrepancy point set, and use $ y_j$ to represent a  component of a $ d$-dimensional point $ \y$. Also, $ \x$ is always used to represent a point in $ \R^{d}$ and $ \y$ is always used to represent a point in $ [0,1]^d$. Denote $1{:}d = \left\{ 1,2,\dots,d\right\}.$ Let $ \uu $ be a subset of $1{:}d$, and $ |\uu|$ be the number of elements in $ \uu$. Let $\boldsymbol{a}_{\uu}{:}\boldsymbol{b}_{-\boldsymbol{u}}$ be  a vector in $ \R^d$ whose $ j$-th component is $ a_j$ if $ j \in \uu$ and $b_j$ otherwise.

	\subsection{Quasi-Monte Carlo methods}
	QMC method is a quadrature rule for approximating the integration of functions over the unit cube $[0,1]^d$. Instead of using random points in Monte Carlo method, QMC method uses deterministic low-discrepancy sequences.
The uniformity of a point set is measured by discrepancy defined below.
\begin{definition}
    For the point set $P = \{\y_1,\dots,\y_n\}$ in the unit cube $ [0,1]^d$, the star discrepancy of $P$ is defined as
    \begin{equation}\notag
        D^*_n(P) := \sup_{B\in \mathcal{B} }\left|\frac{1}{n}\sum_{j = 1}^n \mathbf{1}_B(\y_j) - \lambda(B)\right|,
    \end{equation}
    where the $\lambda(\cdot)$ is the Lebesgue measure and $\mathcal{B}$ is the family of all subintervals in $[0,1]^d$ of the form $\prod_{j =1}^d [0,u_j)\subset [0,1]^d$. 
\end{definition}


In this paper, we work on smooth functions, which are defined below.
\begin{definition}\label{def:smooth}
    
A function $f(\x)$ defined over $K\subseteq \R^d$ is called a \textbf{smooth function} if for any $ \uu \subseteq 1{:}d$, $ \partial^{\uu} f $ is continuous. Let $ \mathcal{S}^d(K)$ be the class of such smooth functions.
\end{definition}

For smooth functions, we have a brief definition for the variation in the sense of Hardy and Krause.
\begin{definition}\label{def:HK}
    If $ f \in \mathcal{S}^d([0,1]^d)$, then the variation of $f$ in the sense of Hardy and Krause is 
    \begin{equation}\notag
        V_{\mathrm{HK}}(f) = \sum_{\varnothing \ne \uu \subseteq 1{:}d}  \int_{\left[0,1\right]^d} \left|\partial^{\boldsymbol{u}}f\left( \boldsymbol{y_u}{:}\mathbf{1}_{-\boldsymbol{u}}\right)\right| d\boldsymbol{y}.
    \end{equation}
\end{definition}

The Koksma-Hlawka inequality \cite{Hlawka1972} provides an error bound for QMC quadrature rule, i.e., 
    \begin{equation}\notag
        \left|\frac{1}{n}\sum_{j = 1}^n f(\boldsymbol{y}_j) - \int _{[0,1]^d} f(\boldsymbol{y}) d\boldsymbol{y} \right| \le V_{\mathrm{HK}}(f) D^*_n\left(\{\boldsymbol{y}_1,\dots,\boldsymbol{y}_n\}\right).
    \end{equation}
Note that the error bound depends on the star discrepancy of the sequence used for QMC methods. There are several kinds of low-discrepancy sequences, such as Sobol' sequece, Halton sequence and Faure sequence (see~\cite{Niederreiter1992,glasserman2004}) with the star discrepancy of $
O(n^{-1}(\log n)^{d})$ for the first $n$ points. Therefore, if the function $f$ is BVHK, then QMC can attain a  rate of convergence $O(n^{-1+\eps})$, where $\eps>0$ is arbitrarily small for hiding the logarithm term. In this paper, we focus on digital nets.
	
\begin{definition}
    An elementary interval of $ [0,1)^{d}$ in base b is an interval of the form
    \begin{equation}\notag
        E = \prod_{j =1}^d \left[\frac{t_j}{b^{k_j}},\frac{t_j+1}{b^{k_j}}\right)
    \end{equation}
    for nonnegative integers $ k_j$ and $ t_j<b^{k_j}$.
\end{definition}
	
\begin{definition}\label{def:tmd-net}
    Let $\lambda,t,m$ be integers with $ m\ge 0,\ 0\le t\le m$, and $ 1\le \lambda < b$. A point set $ \left\{ \y_j\right\}$ of $ \lambda b^m$ points is called a $ \left( \lambda,t,m,d\right)$-net in base b if every elementary interval in base $ b$ of volume $ b^{t-m}$ contains $ \lambda b^{t}$ points of the point set and no elementary interval in base b of volume $ b^{t-m-1}$ contains more than $ b^{t}$ points of the point set.
\end{definition}	

If $ \lambda = 1$, then we use the notation $ (t,m,d)$-net instead. Every $ (t,m,d)$-net in base $ b\ge 2$ is a low-discrepancy point set with $ n = b^m$, whose star discrepancy becomes $
O(n^{-1}(\log n)^{d-1})$. We refer to Niederreiter~\cite{Niederreiter1992}  for more details.

Owen~\cite{owen1995} provided a scrambling method to randomize the QMC points, resulting in a kind of RQMC methods. The randomized points satisfies $ \y_j \sim U[0,1]^d$ and retain the net property. Owen~\cite{owen1997b,owen2008} showed that the scrambled net achieves a convergence rate of $ O(n^{-3/2+\eps})$ for smooth integrands, which is better than the rate $ O(n^{-1+\eps})$ for the unscrambled net. The following two propositions are taken from Owen~\cite{owen1997b}.
\begin{proposition}\label{prop:scr1}
    If $ \left\{ \boldsymbol{a}_j\right\}$ is a $ (\lambda,t,m,d)$-net in base b and $ \left\{ \y_j\right\}$ is the scrambled version of $ \left\{ \boldsymbol{a}_j\right\}$, then $ \left\{ \y_j\right\}$ is a $ (\lambda,t,m,d)$-net in base b with probability $ 1$.
    
\end{proposition}
	
\begin{proposition}\label{prop:scr2}
    Let $ \boldsymbol{a}$ be a point in $ [0,1]^{d}$ and $ \y$ is the scrambled version of $ \boldsymbol{a}$. Then $ \y$ has the uniform distribution on $ \left[ 0,1\right]^d$.
\end{proposition}

\section{Growth conditions}\label{sec:growth}
To drive the convergence rate of RQMC method for unbounded functions over the unit cube $ (0,1)^d$, Owen~\cite{owen2006a} introduced the boundary growth condition for smooth functions defined on $ (0,1)^{d}$,
\begin{equation}\label{eq:owencondition}
    \left| \partial^{\uu}f(\y)\right| \le B\prod_{j =1}^d\min\left( y_j,1-y_j\right)^{-A_j-\mathbf{1}_{j\in \uu}}
\end{equation}
for some $ A_j >0$, $ 0< B< \infty$, and all $ \uu\subseteq 1{:}d$.
By compositing the inverse distribution function, our target integrand is then $ f(\y) = h\circ \Phi^{-1}(\y)$, which is defined in $ (0,1)^{d}$. Therefore, the growth condition for function $ h$ is 
\begin{equation}\label{eq:owengrowth2}
    \bigg|\partial^{\uu}\left(h\circ \Phi^{-1}\right) (\y)\bigg| \le B\prod_{j =1}^d \min\left( y_j,1-y_j\right)^{-A_j - \mathbf{1}_{j\in \uu}}.
\end{equation}

Since $\Phi$ is a specific function, we now work out an equivalent condition on $h(\boldsymbol{y})$ for ensuring \eqref{eq:owengrowth2}.
Let $ \x = \Phi^{-1}(\y)$. Then the left hand side of~\eqref{eq:owengrowth2} is 
\begin{equation}\notag
    \bigg|\partial^{\uu}\left(h\circ \Phi^{-1}\right) (\y)\bigg| = \bigg|\partial^{\uu}h(\x)\prod_{j\in \uu} \frac{d\Phi^{-1}(y_j)}{dy_j}\bigg| = \bigg|\partial^{\uu}h(\x)\prod_{j\in \uu} \frac{1}{\varphi(x_j)}\bigg|,
\end{equation}
where $ \varphi(x) = \frac{1}{\sqrt{2\pi}}\eexp{-x^2/2}$ is the Gaussian probability density. Since
\begin{equation}\notag
\begin{aligned}
     \min\left( y_j,1-y_j\right) &= \min\left\{ \Phi(x_j),1-\Phi(x_j)\right\} = 1-\Phi(|x_j|),
\end{aligned}
\end{equation}
the inequality \eqref{eq:owengrowth2} is equivalent to
\begin{equation}\notag
    \bigg|\partial^{\uu}h\left( \x\right)\bigg| \le B \prod_{j=1}^d\left( 1-\Phi(|x_j|)\right)^{-A_j - \mathbf{1}_{j\in \uu}} \prod_{j \in \uu} \varphi(x_j): = G(\x).
\end{equation}
Note that for any $ \eps >0$, if $ |x_j|$ is large enough then 
\begin{equation}
    \eexp{x_j^2/2} \le \left( 1-\Phi(|x_j|)\right)^{-1} \le \eexp{(1+\eps)x_j^2/2} \notag.
\end{equation}
Consequently, for any $ \eps > 0$, there exist $ B_1$ and  $ B_2$ such that
\begin{align}\label{eq:owencondition_in_R}
    B_1\eexp{\sum_{j=1}^d \frac{A_j}{2}x_j^2} \le G(\x) \le B_2 \eexp{\sum_{j=1}^d\frac{A_j+\eps}{2}x_j^2}.
\end{align}

As a result, the inequality \eqref{eq:owengrowth2} implies $|\partial^{\uu}h( \x)|=O(\exp\{\sum_{j=1}^d\frac{A_j+\eps}{2}x_j^2\})$. On the other hand, if $|\partial^{\uu}h( \x)|=O(\exp\{\sum_{j=1}^d\frac{A_j}{2}x_j^2\})$, then the inequality \eqref{eq:owengrowth2} holds, leading to a mean error rate of $ O(n^{-1+\max_j A_j+\eps})$ for RQMC \cite{owen2006a}. To achieve the convergence rate $ O(n^{-1+\eps})$, all $ A_j$ must be arbitrarily small. This leads to the definition of the ``QMC-friendly" condition.

\begin{definition}
    We say that $ h(\boldsymbol{x})$ satisfies the ``QMC-friendly" condition if for any fixed $M>0$, there exists $ B>0 $ such that $$\sup_{\uu \subseteq 1{:}d} \left|\partial^{\boldsymbol{u}}h(\boldsymbol{x})\right| \le Be^{M|\x|^2}.$$ 
\end{definition}


We next refine the growth conditions in Owen~\cite{owen2006a} and define some slower-growing function classes (relatively smaller), and use the P-QMC method to obtain better convergence results. 

\begin{definition}\label{def:growth class}
For $ M>0,\ B>0$ and $ k>0$, define polynomial growth class,
    \begin{equation}\notag
        G_p(M,B,k):= \left\{h \in \mathcal{S}^d(\R^d)  :\sup_{\uu \subseteq 1{:}d} \left|\partial^{\boldsymbol{u}}h(\boldsymbol{x})\right| \le M|\x|^k+B\right\} ,
    \end{equation}
    and define exponential growth class,
    \begin{equation}\notag
        G_e(M,B,k):= \left\{h\in \mathcal{S}^d(\R^d) :\sup_{\uu \subseteq 1{:}d} \left|\partial^{\boldsymbol{u}}h(\boldsymbol{x})\right| \le Be^{M|\x|^k}\right\},
    \end{equation}
    where $\mathcal{S}^d(\R^d)$ is the class of smooth functions given in Definition~\ref{def:smooth}.
    We say that $h$ has {\bf polynomial growth} if there exists $M>0, B>0, k>0$ such that $h\in G_p(M,B,k)$, and $h$ has {\bf exponential growth of order $ k$} if there exists $M>0, B>0$ such that $h \in G_e(M,B,k)$. 
\end{definition}

When the order $ k =2$, there are functions in the growth class that grow too fast to satisfy the ``QMC-friendly" condition. This leads to the definition of fast growth class.
\begin{definition}\label{def:fast}
    If $ M < 1/2,\ B>0 $, then we call $ G_e(M,B,2) $ the \textbf{fast growth class}. We say that $ h$ has  \textbf{fast growth}, if $ h$ belongs to the fast growth class.
\end{definition}

Note that in the definition of fast growth class, we have reached the minimum restriction on $ M$. Because if $ M\ge 1/2$, there are cases where the integral is infinity. And when $ M \ge 1/4$, the variance can be infinity, therefore, the MC method does not converge.

 The growth conditions we have defined are a further subdivision of Owen's growth condition. If $ h$ satisfies Owen's growth condition~\eqref{eq:owengrowth2}, then there are three situations,
\begin{enumerate}[label = \Roman*),align = left]
  \item $ h$ has polynomial growth;\label{condition:poly}
  \item $ h$ has exponential growth of order $ k<2$;\label{condition:exp}
  \item $ h$ has fast growth.\label{condition:fast}
\end{enumerate}

Note that \ref{condition:poly} and \ref{condition:exp} constitute the ``QMC-friendly" condition. With the help of the projection operator defined later, if $ h$ satisfies condition \ref{condition:poly} or \ref{condition:exp}, P-QMC and RQMC method will yield better results. 

In option pricing, under the well known Black-Scholes model \cite{glasserman2004}, the payoff functions can be written as 
\begin{equation}\notag
    h(W) = f(\eexp{AW}),
\end{equation}
where $ A \in \R^{d\times d}$ is a constant matrix. If $ f$ and its derivatives up to order $ d$ have polynomial growth, then by Theorem~\ref{them:composition} in Appendix, we obtain $ h \in G_e(M^{\prime},B,1)$ for some $ M^{\prime} >0$ and $ B>0$. However, the function $f$ usually has kinks or jumps in option pricing. To reclaim the efficiency of QMC, it was suggested to use the conditioning method to smooth the integrand to obtain a new function  $ f \in C^{d-1}(\R^{d-1})$. We refer to \cite{griewank2018high,zhang2020} for more details. From this perspective, our findings in this paper can be applied  for many problems in financial engineering.

Now, we focus on the situation where $ h$ satisfies the fast growth. By relation~\eqref{eq:owencondition_in_R}, if $ h \in G_e(M,B,2) $ then $ h $ satisfies~\eqref{eq:owengrowth2} with $ \max_j A_j = 2M$, and therefore, by the results of Owen~\cite{owen2006a}, RQMC achieves a mean error rate of $ O\left( n^{-1+2M+\eps}\right)$. From this standpoint, $ M<1/2$ is a very weak constraint; otherwise, it would not converge. In Section~\ref{sec:is}, we will prove that with a suitable IS, the P-QMC method achieves the convergence rate of $ O(n^{-1+\eps}) $ and the RQMC method achieves the convergence rate of $ O(n^{-3/2+\eps})$ for $ M <1/2$. As a result, we can conclude that using IS accelerates the convergence rate of QMC.

\section{The P-QMC method}\label{sec:pqmc}

In this section, we propose a method, called projection based quasi-Monte Carlo method (P-QMC). After composite the projection operator and inverse distribution function, the modified integrand defined on the unit cube $ \left[ 0,1\right]^d$ is of BVHK, and therefore error analysis can be achieved when using QMC method to the modified integrand.

\subsection{Projection operator}
We introduce the concept of a projection operator, which projects the $\R^d$ space to a bounded region $ \left[ -R,R\right]^d$ for a constant $ R>0$.
\begin{definition}\label{def:projection}
    If $x \in \R $, for any $R > 0 $, we define the projection operator $\widehat P_R: \R \rightarrow \R $ as 
    \begin{equation}\notag
        \widehat P_R(x) = {\rm arg}\min_{y \in [-R,R]} |y-x| = \left\{ \begin{array}{cl}
            x, & x \in \left[ -R,R\right] \\
            R, & x > R\\
            -R, & x< -R
        \end{array}\right..
    \end{equation}
    If $\x = \left( x_1,\dots,x_d\right) \in \R^d$, then $\widehat P_R :\R^d\rightarrow\R^d$ acts on each component of $\x$, i.e., $\widehat P_R(\x)$ $=$ $(\widehat P_R(x_1),\ldots,\widehat P_R(x_d))$.
\end{definition}

From the definition of $\widehat P_R$, we can see that $\widehat P_R$ is a function from $\R^d$ to the cube $[-R,R]^d$ and that it projects each component of a vector in $ \R^d$ from $ \R $ to $ [-R,R]$. The operator $\widehat P_R$ is continuous but not differentiable at some points. We next modify $\widehat P_R$ to make it smooth. The following lemma gives us a modification for the one-dimensional case, and the multidimensional case can be defined in the component-wise way.
\begin{lemma}\emph{(One-dimensional case)}\label{lemma:projection}
    For any $\eps\in (0,R)$, there exists a modification of the projection operator $\widehat P_R$, denoted as $P_R$ , which is defined on $[-\infty,\infty]$ and satisfies
    \begin{enumerate}[label=\roman*),align = left]
        \item $P_R(x) = \widehat P_R(x) = x \ \ for\  x \in [-R+\eps,R-\eps]$;\label{condition:proj1}
        \item $P_R $ has continuous derivative of order 1, i.e., $P_R \in C^1(\R)$;\label{condition:proj2}
        \item $|\frac{d P_R(x)}{dx}|\le \mathbf{1}_{\{|x|\le R\}}$;\label{condition:proj3}
        \item $|x|\mathbf{1}_{\{|x|\le R-\eps\}}+(R-\eps)\mathbf{1}_{\{|x|> R-\eps\}}\le |P_R(x)| \le |x|\mathbf{1}_{\{|x|\le R\}}+ R\mathbf{1}_{\{|x| > R\}}$;\label{condition:proj4}
        \item $ P_R(\infty) = \lim_{x\rightarrow \infty} P_R(x) $ and $ P_R(-\infty) = \lim_{x\rightarrow -\infty} P_R(x) $.\label{condition:proj5}
    \end{enumerate}
\end{lemma}
\begin{proof}
    There are many constructive methods for proving this lemma. For instance, we adopt a quadratic function to replace the part that $P_R \ne \widehat P_R$, so that the piecewise function constructed is smooth. Define $P_R$ as:
    \begin{equation}\label{eq:smoothprojection}
        P_R(x) = \left\{\begin{array}{ll}
            -R+\frac{\eps}{2}, \ &\ x\in [-\infty,-R]\\
            \frac{1}{2\eps}x^2 + \frac{R}{\eps}x +\frac{(R-\eps)^2}{2\eps},
            \ &\ x\in(-R,-R+\eps)\\
            x,\ &\ x\in [-R+\eps,R-\eps]\\
            -\frac{1}{2\eps}x^2 + \frac{R}{\eps}x - \frac{(R-\eps)^2}{2\eps},\ &\ x\in (R-\eps,R)\\
            R-\frac{\eps}{2},\ &\ x\in [R,\infty]
        \end{array}\right. .
    \end{equation}
    It is easy to verify that the $P_R$ satisfies the conditions in Lemma~\ref{lemma:projection}.
\end{proof}
\begin{figure}[htbp]
  \centering
  \includegraphics[width=0.7\textwidth]{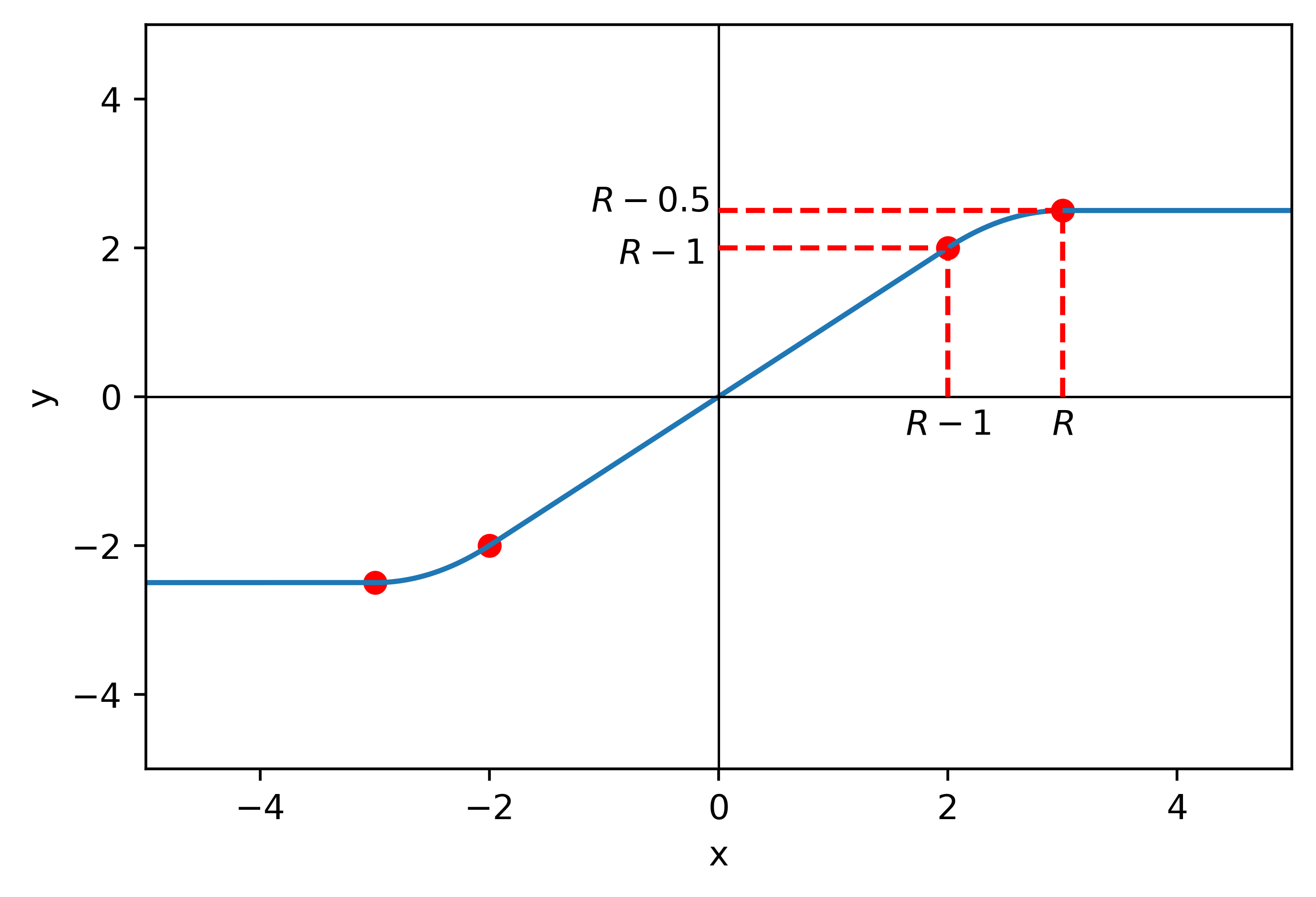}
  \caption{One-dimensional smoothed projection operator with $\eps = 1$}
  \label{fig:projection}
\end{figure}

From now on, we will call $ P_{R}$ the \textbf{projection operator} if each component of $ P_R $ satisfies the conditions in Lemma~\ref{lemma:projection} and with a slight abuse of terminology, call $ R $ the \textbf{projection radius}. For convenience, we take $\eps = 1 $ in Lemma~\ref{lemma:projection} and drop the dimension notation if there is no misunderstanding.  Figure~\ref{fig:projection} shows the one-dimensional smoothed projection operator with $\eps = 1$.

\begin{proposition}\label{prop:projection}
    Denote $ H = \left[ -R+1,R-1\right]^d$. If $ \x \in H$, then $ \x = P_R(\x)$. If $ \x \in \R^d \setminus H$, then for any point $ \boldsymbol{\xi}$ lies on the line segment connecting points $ \x$ and $ P_R(\x)$, we have $ |\boldsymbol{\xi}| \ge R-1$.
\end{proposition}
\begin{proof}
    It suffices to prove the second statement. There is a $ 0\le \lambda \le 1$, such that $ \boldsymbol{\xi} = (1-\lambda)\x + \lambda P_R(\x)$. Let $ \boldsymbol{\xi} = (\xi_1,\dots,\xi_d)$. For any $ \x \in \R^d \setminus \left[ -R+1,R-1\right]^d$, there is an index $j\in 1{:}d$ such that $ |x_j| \ge |P_R(x_j)| \ge R-1$. Therefore, $$ |\boldsymbol{\xi}| \ge |\xi_j|  = |(1-\lambda)x_j + \lambda P_R(x_j)| = (1-\lambda)|x_j| + \lambda |P_R(x_j)|\ge R-1 ,$$
    where the second equality holds because $ x_j$ and $ P_R(x_j)$ have the same sign.
\end{proof}


\subsection{Error analysis for the P-QMC method}

In this section, we discuss how to modify the integrand through the projection operator $ P_R$ to achieve BVHK and prove the corresponding convergence property. To this end, we modify the integrand $h$ by compositing the projection operator $P_R$, and the integrand changes from $ h\circ \Phi^{-1}$ to the modified integrand $ h\circ P_R \circ \Phi^{-1}$. The following lemma shows that the modified integrand is a smooth function on $ [0,1]^d$.
\begin{lemma}\label{lemma:smooth}
    If $ h\in \mathcal{S}^d(\R^d)$, then the modified integrand $ h\circ P_R \circ \Phi^{-1}\in \mathcal{S}^d([0,1]^d)$.
\end{lemma}
\begin{proof}
    Note that our projection operator $ P_R $ is defined on $ [-\infty,\infty]^d$, so the modified integrand $ h\circ P_R \circ \Phi^{-1}$ is well-defined on $ [0,1]^d$. For any $ \uu \subseteq 1{:}d$,
    \begin{align}
        \partial^{\boldsymbol{u}} \left(h\circ P_R \circ \Phi^{-1}(\y)\right) &= \left( \partial^{\boldsymbol{u}} h\right)\circ P_R \circ \Phi^{-1}(\y) \cdot \prod_{j\in \uu} \frac{d P_R(\Phi^{-1}(y_{j}))}{d \Phi^{-1}(y_{j})}\frac{d \Phi^{-1}(y_{j})}{d y_{j}}.\notag
    \end{align}
    The right hand side of the above equality is well-defined on $ [0,1]^d$ because the projection operator $ P_R$ satisfies condition~\ref{condition:proj3} of Lemma~\ref{lemma:projection} and the derivative vanishes when $ y_j$ approaches $ 0$ or $ 1$. Therefore, it is continuous on $ [0,1]^d$.
\end{proof}

To deal with the singularities, Owen~\cite{owen2006a} used the low variation extension due to Sobol',
\begin{equation}\notag
            \widehat f(\y) = f(\boldsymbol{c} ) + \sum_{\uu \ne \varnothing} \int_{[\boldsymbol{c}_{\uu},\y_{\uu}]} \mathbf{1}_{\left\{\boldsymbol{z}_{\uu}{:}\boldsymbol{c}_{-\uu}\in K\right\}} \partial^{\uu} f(\boldsymbol{z}_{\uu}{:}\boldsymbol{c}_{-\uu}) d\boldsymbol{z}_{\uu},
        \end{equation}
where $K\subseteq [0,1]^d$ is Sobol's extensible with anchor $\boldsymbol{c}$. 
This extension is not differentiable at some points,  implying $\widehat f\notin \mathcal{S}^d([0,1]^d)$. Our modified integrand $ h\circ P_R \circ \Phi^{-1}$ is smooth, allowing to calculate the variation in the sense of Hardy and Krause simply by Definition~\ref{def:HK}. Furthermore, we can apply the results in Owen~\cite{owen2008}  to achieve higher convergence rate for such a smooth function.

Our P-QMC method uses the modified quadrature method 
\begin{equation}
    \widehat I_n^R(h) := \frac{1}{n} \sum_{j = 1}^n h\circ P_R \circ \Phi^{-1} (\boldsymbol{y}_j)\notag ,
\end{equation}
where $ \left\{ \y_1,\dots,\y_n\right\}$ is a low-discrepancy point set. We claim that the modified method overcomes the problem of infinite variation in the sense of Hardy and Krause, and moreover, its convergence rate is better than those in Owen~\cite{owen2006a} and He et al.~\cite{he2023}.

The total error of the proposed P-QMC method can be decomposed into the following two parts
\begin{equation}\label{eq:errordecom}
    \bigg| \widehat I_n^R(h) - \E\left[ h(W)\right]\bigg| \le \underbrace{\bigg| \widehat I_n^R(h) - \E\left[ h\circ P_R(W)\right]\bigg|}_\text{QMC error} + \underbrace{ \E \bigg|h\circ P_R(W)- h(W)\bigg|}_\text{projection error}.
\end{equation}
We are going to bound the error $  \E\left| h\circ P_R(W) -  h(W)\right|$ due to using the projection operator, which is called the \textbf{projection error}. Then, for the \textbf{QMC error} $ \left| \widehat I_n^R(h) - \E\left[ h\circ P_R(W)\right]\right|$, we calculate the variation of the modified integrand in the sense of Hardy and Krause, and use the Koksma-Hlawka inequality to obtain an upper bound. Both of these errors depend on the projection radius $ R$. Finally, we choose an appropriate $ R$ to balance the two sources of errors and obtain the convergence rate for the total error.

In contrast, if we use the original estimator $ \widehat I_n(h)$ defined in~\eqref{eq:qmcapproximate}, we will have an additional term called \textbf{sample error}, as shown below
\begin{equation}
    \bigg| \widehat I_n(h) - \E\left[ h(W)\right]\bigg| \le \underbrace{\bigg|  \widehat I_n(h) - \widehat I_n^R(h)\bigg|}_\text{sample error} + \left| \widehat I_n^R(h) - \E\left[ h(W)\right]\right|.
\end{equation}
The magnitude of sample error depends on how far the low-discrepancy points we use deviate from the boundaries. In this regard, Owen~\cite{owen2006a} specifically studied the case of Halton sequences and gave several regions that avoid the origin and other corners so that the sample error vanishes. Calculating sample error is generally a cumbersome task for the P-QMC method. However, if we use the RQMC method and consider the expected error, the sample error can be bounded by the projection error, thereby avoiding the estimation of this term. The next lemma will be used several times in the following proofs.
\begin{lemma}\label{lemma:intestimate}
    If $ d\in \N_{+}, M>0$ and $ T \ge 1/\sqrt{2M}$, then
    \begin{equation}\label{eq:basisint}
        \int_T^{\infty} x^{d}\eexp{-Mx^2} dx \le \frac{(d+2)!!}{4M}T^{d-1}\eexp{-MT^2}.
    \end{equation}
\end{lemma}

\begin{proof}
    With $ u = \sqrt{2M}x$, the left hand side of~\eqref{eq:basisint} is equal to 
    \begin{align}
        &\left( 2M\right)^{-\frac{d+1}{2}}\int_{\sqrt{2M}T}^{\infty} u^{d}e^{-u^2/2}du \notag\\
         =& \left( 2M\right)^{-\frac{d+1}{2}}\left( \left( \sqrt{2M}T\right)^{d-1}e^{-MT^2}+(d-1)\int_{\sqrt{2M}T}^{\infty}u^{d-2}e^{-u^2/2}du\right)\notag\\
         \le &\left( 2M\right)^{-\frac{d+1}{2}}\left( ( \sqrt{2M}T)^{d-1}+(d-1)(\sqrt{2M}T)^{d-3}+\cdots+(d-1)!!\right)\eexp{-MT^2}\notag\\
         \le &\left( 2M\right)^{-\frac{d+1}{2}}\frac{(d+2)}{2}(d-1)!!(\sqrt{2M}T)^{d-1}\eexp{-MT^2}\notag\\
         \le &\frac{(d+2)!!}{4M} T^{d-1} \eexp{-MT^2}\notag,
    \end{align}
    where in the first inequality we use 
    \begin{equation}\notag
        \int_{\sqrt{2M}T}^{\infty} e^{-u^2/2} du \le \int_{\sqrt{2M}T}^{\infty} u e^{-u^2/2} du = \eexp{-MT^2}
    \end{equation}
    for the case $ d $ is even, and in the last inequality, we use $ (d-1)!! \le d!!$.
\end{proof}

 When $ h$ falls into different growth classes, there are corresponding worst case estimates for the projection error. Note that the mean error can be bounded by the root mean squared error.
\begin{lemma}\label{lemma:projectioninterror}
   Assume  $ M>0$ and  $B>0$. For polynomial growth class $ G_p(M,B,k)$ with $ k>0$, if $ R>2$, then
    \begin{equation}
        \sup_{h\in G_p(M,B,k)} \E\left[\left(h\left(W\right)-h\circ P_R(W)\right)^2\right] \le C_1(R-1)^{[2k]+d-1}\eexp{-\frac{1}{2}(R-1)^2}\label{estimate:poly},
    \end{equation}
    where $ C_1 = (\frac{\pi}{2})^{\frac{d}{2}-1}\left(M^2+B^2\right)([2k]+d+2)!!$ and $ [2k] $ is the integer part of $ 2k$. And for exponential growth class $ G_e(M,B,k)$ with $ 0<k<2$, if $ R>1+\sqrt{2}$, then
    \begin{equation}
        \sup_{h\in G_e(M,B,k)} \E\left[\left(h\left(W\right)-h\circ P_R(W)\right)^2\right]  \le C_2(R-1)^{d-1}\eexp{-\frac{1}{4}(R-1)^2}\label{estimate:exp},
    \end{equation}
    where $ C_2 = (\frac{\pi}{2})^{\frac{d}{2}-1}B^2\eexp{(4kM)^{\frac{k}{2-k}}}(d+2)!!$.
\end{lemma}
\begin{proof}
    We first prove~\eqref{estimate:poly}. For any $h\in G_p(M,B,k)$,
    \begin{align}
         \E\left[(h\left(W\right)-h\circ P_R(W))^2\right] 
        &\le (2\pi)^{-\frac{d}{2}}\int_{|\x|\ge R-1} \left(\left|h(\x)\right|+\left|h\circ P_R(\x)\right|\right)^2 e^{-|\x|^2/2}
       d\x\notag\\
        & \le (2\pi)^{-\frac{d}{2}}\int_{|\x|\ge R-1} \left(2M|\x|^k+2B\right)^2 e^{-|\x|^2/2} d \x \label{eq:lemma2-1}\\
        & \le (2\pi)^{-\frac{d}{2}}\int_{|\x|\ge R-1} 8\left(M^2|\x|^{2k}+B^2\right)e^{-|\x|^2/2}d \x,\notag
    \end{align}
    where we use $|h(\x)|\le M|\x|^{k}+B $ and $ |h\circ P_R(\x)| \le M|P_R(\x)|^k+B\le M|\x|^{k}+B$ by Lemma~\ref{lemma:projection}. Using the polar coordinates transformation, the right hand side of ~\eqref{eq:lemma2-1} can be written as
    \begin{align}
        \frac{8}{(2\pi)^{\frac{d}{2}}}&\int_0^{2\pi}\cdots\int_0^\pi\prod_{j =1}^{d-2}|\sin\psi_j|^{d-1-j}\int_{ R-1}^{\infty} |\x|^{d-1}\left(M^2|\x|^{2k}+B^2\right)e^{-\frac{|\x|^2}{2}}d|\x|d\psi_{1{:}(d-1)}\notag\\
        &\le 2\left(\frac{\pi}{2}\right)^{\frac{d}{2}-1}\int_{ R-1}^{\infty}\left( M^2|\x|^{2k+d-1}+B^2|\x|^{d-1}\right)e^{-|\x|^2/2}d|\x|\notag\\
        &\le 2\left(\frac{\pi}{2}\right)^{\frac{d}{2}-1}\int_{ R-1}^{\infty}\left(M^2+B^2\right)|\x|^{[2k]+d}e^{-|\x|^2/2}d|\x|\notag\\
        &\le \left(\frac{\pi}{2}\right)^{\frac{d}{2}-1}\left(M^2+B^2\right)([2k]+d+2)!!(R-1)^{[2k]+d-1}\eexp{-\frac{1}{2}(R-1)^2}\notag ,
    \end{align}
    where $ [k] $ is the the integer part of $ k$, at the first inequality we use the fact that $|\sin\psi|^j  \le 1$ and at the second inequality we use  $ |\x|\ge 1$ and $ 2k-1\le [2k] $, and at the last inequality, we use Lemma~\ref{lemma:intestimate}.

    For the proof of~\eqref{estimate:exp}, we start at~\eqref{eq:lemma2-1}. Replacing the term $ 2M|\x|^k+2B$ with $ 2B\eexp{M|\x|^k}$ and using the same method, we obtain
    \begin{align}
        \E\left[(h\left(W\right)-h\circ P_R(W))^2\right]  &\le \left(\frac{\pi}{2}\right)^{\frac{d}{2}-1}\int_{ R-1}^{\infty} B^2|\x|^{d-1}\eexp{2M|\x|^k}e^{-|\x|^2/2}d|\x| \notag\\
        &\le \left(\frac{\pi}{2}\right)^{\frac{d}{2}-1}B^2\int_{R-1}^{\infty}|\x|^d\eexp{2M|\x|^k}e^{-|\x|^2/2}d|\x|\label{eq:lemma2-3} ,
    \end{align} 
    where we use the fact that $|\x| \ge 1$. Note that by Young's inequality, for any $\eps > 0 $
    \begin{equation}\notag
       |\x|^k \le \frac{\left(\eps|\x|^k\right)^{\frac{2}{k}}}{\frac{2}{k}} + \frac{\left(\frac{1}{\eps}\right)^{\frac{2}{2-k}}}{\frac{2}{2-k}} = \frac{k}{2}\eps^{\frac{2}{k}}|\x|^{2} + \frac{2-k}{2\eps^{\frac{2}{2-k}}}  .
    \end{equation}
    Choose $\eps = (4kM)^{-\frac{k}{2}}$, so that 
    \begin{equation}\notag
        2M|\x|^k \le \frac{1}{4}|\x|^2 + (4kM)^{\frac{k}{2-k}} .
    \end{equation}
    Therefore, the right hand side of~\eqref{eq:lemma2-3} is bounded by 
    \begin{align}
        \left(\frac{\pi}{2}\right)^{\frac{d}{2}-1}&B^2\eexp{(4kM)^{\frac{k}{2-k}}}\int_{R-1}^{\infty} |\x|^d\eexp{-\frac{|\x|^2}{4}}d|\x| \notag\\
        &\le \left(\frac{\pi}{2}\right)^{\frac{d}{2}-1}B^2\eexp{(4kM)^{\frac{k}{2-k}}}(d+2)!!(R-1)^{d-1}\eexp{-\frac{1}{4}(R-1)^2}\notag,
    \end{align} 
    where in the above inequality, we use Lemma~\ref{lemma:intestimate}.
\end{proof}


The key to bound the QMC error lies in calculating the variation of the modified integrand $h\circ P_R \circ \Phi^{-1} $ in the sense of Hardy and Krause. For $ h$ in different growth classes, corresponding upper bound estimates for the variation are given in the following lemma.

\begin{lemma}\label{lemma:HK}
    For every fixed $ M>0, B>0, k>0$ and projection radius $ R>0$,
    \begin{equation}\label{HK:poly}
        \sup_{h \in G_p(M,B,k)} V_{\mathrm{HK}}(h\circ P_R \circ \Phi^{-1}) \le 2^{2d}\left(Md^{\frac{k}{2}}R^{k+d}+BR^d\right),
    \end{equation}
    \begin{equation}\label{HK:exp}
        \sup_{h \in G_e(M,B,k)} V_{\mathrm{HK}}(h\circ P_R \circ \Phi^{-1}) \le 2^{2d}BR^d\eexp{M(\sqrt{d}R)^k}.
    \end{equation}
\end{lemma}
\begin{proof}
     For the proof of~\eqref{HK:poly}, note that by Lemma~\ref{lemma:smooth}, the modified integrand $h\circ P_R \circ \Phi^{-1} \in \mathcal{S}^d([0,1]^d)$. Therefore, we can use the definition of the variation in the sense of Hardy and Krause for smooth functions, i.e.,
    \begin{align}
        V_{\mathrm{HK}}(h\circ P_R \circ \Phi^{-1}) = \sum_{\varnothing \ne \uu \subseteq 1{:}d} \int_{\left[0,1\right]^d} \left|\partial^{\boldsymbol{u}}h\circ P_R \circ \Phi^{-1}\left( \boldsymbol{y_{u}}{:}\mathbf{1}_{-\boldsymbol{u}}\right)\right| d\boldsymbol{y}\notag.
    \end{align}
    It suffices to analyze every term in the right hand side of the above equality. Recall the relation~\eqref{eq:distribute} for distribution function $\Phi$. 
    \begin{align}
         &\bigg|\partial^{\boldsymbol{u}}h\circ P_R \circ \Phi^{-1}\left( \boldsymbol{y_{u}}{:}\mathbf{1}_{-\boldsymbol{u}}\right)\bigg| \notag\\
         =& \bigg| \partial^{\boldsymbol{u}}h\left(P_R\left(\boldsymbol{x}_{\boldsymbol{u}}\right){:}P_R \circ \Phi^{-1}\left(\mathbf{1}_{-\boldsymbol{u}}\right)\right)\bigg|\cdot\left|\prod_{j \in \uu}\frac{d P_R(\Phi^{-1}(y_{j}))}{d \Phi^{-1}(y_{j})}\frac{d \Phi^{-1}(y_{j})}{d y_{j}}\right|  \notag\\
         &\le \left(Md^{\frac{k}{2}}|R|^k+B\right)\prod_{j \in \uu}\mathbf{1}_{\{|x_{j}|\le R\}}\sqrt{2\pi}\eexp{\frac{|x_{j}|^2}{2}}\label{eq:HKestimate1},
    \end{align}
    where we use $h\in G_p(M,B,k)$ and the conditions~\ref{condition:proj3} and~\ref{condition:proj4} in Lemma~\ref{lemma:projection} for the smooth projection operator $P_R$. Integrating both sides of \eqref{eq:HKestimate1} and changing the variable from $\boldsymbol{y}$ to $\boldsymbol{x}$, we have
    \begin{align}
        \int_{\left[0,1\right]^d} &\left|\partial^{\boldsymbol{u}}h\circ P_R \circ \Phi^{-1}\left( \boldsymbol{y_{u}};\mathbf{1}_{-\boldsymbol{u}}\right)\right| d\boldsymbol{y} \notag\\
        &\ \ \ \le \left(Md^{\frac{k}{2}}R^k+B\right) \prod_{j \in \uu}\int_{-R}^Rdx_{j}\prod_{j\in\overline{\uu}}\int_{-\infty}^{\infty}\frac{1}{\sqrt{2\pi}}\eexp{-\frac{|x_j|^2}{2}}d x_j\label{HK:estimate2}\\ 
        &\ \ \ \le 2^d\left(Md^{\frac{k}{2}}R^{k+d}+BR^d\right)\notag,
    \end{align}
    where $ \overline{\uu}$ is the complement of $ \uu$. It follows from the above inequality that
    \begin{equation}
        V_{\mathrm{HK}}(h\circ P_R \circ \Phi^{-1}) \le 2^{2d}\left(Md^{\frac{k}{2}}R^{k+d}+BR^d\right)\notag .
    \end{equation}
    Therefore, ~\eqref{HK:poly} holds. For the proof of ~\eqref{HK:exp}, it suffices to replace the $Md^{\frac{k}{2}}R^k+B$ in~\eqref{HK:estimate2} with $B\eexp{M(\sqrt{d}R)^k}$.
\end{proof}

The projection operator plays a crucial role in computing the variation of functions in the sense of Hardy and Krause. Through compositing the projection operator, the non-zero region of the mixed partial derivative of the modified integrand is restricted by $R$, so the variation in the sense of Hardy and Krause can be bounded. By combining the two lemmas above, we obtain the following error bounds.
\begin{theorem}\label{them:pm-qmc}
    Let $\{\y_1,\dots,\y_n\}$ be a low-discrepancy point set.  The P-QMC method to approximate the integral $ \E\left[ h(W)\right]$ is
    \begin{equation}\label{eq:p_esitmator}
        \widehat I_n^R(h) := \frac{1}{n}\sum_{j =1}^n h\circ P_{R} \circ \Phi^{-1}(\boldsymbol{y}_j).
    \end{equation}
    (\uppercase\expandafter{\romannumeral1}) For polynomial growth class $ G_p(M,B,k)$, by choosing $R = \sqrt{4\log n}+1$, we obtain
        \begin{equation}\label{eq:thempoly}
            \sup_{h\in G_p(M,B,k)}\left|\widehat I_n^R(h) - \E\left[h(W)\right]\right| = O\left(n^{-1}(\log n)^{\frac{3d}{2}+\frac{k}{2}-1}\right) .
        \end{equation}
    (\uppercase\expandafter{\romannumeral2}) For exponential growth class $ G_e(M,B,k)$ with order $ 0<k<2$, by choosing $R = \sqrt{8\log n}+1$, we obtain
        \begin{align}
            \sup_{h\in G_e(M,B,k)}\left|\widehat I_n^R(h) - \E\left[h(W)\right]\right| 
            = O\left(n^{-1}(\log n)^{\frac{3d}{2}-1}\eexp{M(8d\log n)^{\frac{k}{2}}}\right)\label{eq:themexp} .
        \end{align}

\end{theorem}
\begin{proof}
    For the proof of \eqref{eq:thempoly}, note that
    \begin{align}
        &\bigg|\widehat I_n^R(h) - \E\left[h(W)\right]\bigg| 
        \le \bigg|\widehat I_n^R(h) -\E\left[h\circ P_R(W)\right]\bigg|+\E\bigg|h\circ P_R(W)- h(W)\bigg|\notag\\
        \le & V_{\mathrm{HK}}(h\circ P_R \circ \Phi^{-1})D^{*}_n\left(\{\y_1,\dots,\y_n\}\right) + \sqrt{C_1}(R-1)^{\frac{[2k]+d-1}{2}}\eexp{-\frac{1}{4}(R-1)^2}\label{eq:twoerror}\\
        \le & 2^{2d}C\left(Md^{\frac{k}{2}}R^{k+d}+BR^d\right)\frac{(\log n)^{d-1}}{n} +\sqrt{C_1}(R-1)^{\frac{[2k]+d-1}{2}}\eexp{-\frac{1}{4}(R-1)^2} , \label{eq:ineq}
    \end{align}
    where $ C$ and $ C_1$ are constants. The first term of~\eqref{eq:twoerror} is obtained by Koksma-Hlawka inequality and the second term is due to~\eqref{estimate:poly} in Lemma~\ref{lemma:projectioninterror}. To make magnitude of the both term of~\eqref{eq:ineq} approximately the same,  we choose $R = \sqrt{4\log n}+1$, and then the second term is $ O(n^{-1}(\log n)^{([2k]+d-1)/4})$ and the first term is $ O(n^{-1}(\log n)^{3d/2+k/2-1})$. Therefore, by taking supreme at the both sides of~\eqref{eq:ineq}, we obtain~\eqref{eq:thempoly}. The proof for~\eqref{eq:themexp} is straight forward.
\end{proof}

\begin{Remark}
    In both cases of Theorem~\ref{them:pm-qmc}, the convergence rate is $ O(n^{-1+\eps})$, and all we need to notice is that for any $\eps > 0 $ and $0<k<2$, with $ t = \log n$,
    \begin{equation}
        \lim_{n\rightarrow \infty} \frac{\eexp{M\left(\sqrt{8d\log n}\right)^k}}{n^{\eps}}  =  \lim_{t \rightarrow \infty} \eexp{M(8d)^{\frac{k}{2}}t^{\frac{k}{2}}-\eps t} = 0\notag .
    \end{equation}
    However, the results in Theorem~\ref{them:pm-qmc} offer a more detailed analysis compared to the results presented in Owen~\cite{owen2006a}. These findings indicate that as the function grows faster, the performance of the QMC method becomes worse. 
\end{Remark}

\begin{corollary}\label{cor:pqmc}
    Let $ \left\{ \y_1,\dots,\y_n\right\}$ be an RQMC point set used in the estimator $\widehat I_n(h)$ given by \eqref{eq:qmcapproximate} such that each $ \y_j \sim U[0,1]^d $ and 
    \begin{equation}\notag
        \E\left[ D^{*}_n\left(\{\y_1,\dots,\y_n\}\right)\right] \le C \frac{(\log n)^{d-1}}{n},
    \end{equation}
    where $ C $ is a constant independent of $n$.\\
    (\uppercase\expandafter{\romannumeral1}) For the polynomial growth class $G_p\left( M,B,k\right)$ with $ k>0$,
    \begin{equation}\notag
        \sup_{h\in G_p(M,B,k)}\E\left[\left(\widehat I_n(h) - \E[h(W)]\right)^2\right] = O\left(n^{-2}(\log n)^{3d+k-2}\right) .
    \end{equation}\\
    (\uppercase\expandafter{\romannumeral2}) For the exponential growth class $G_e\left( M,B,k\right)$ with order $ 0<k<2$, 
    \begin{equation}\notag
        \sup_{h\in G_e(M,B,k)}\E\left[\left(\widehat I_n(h) - \E[h(W)]\right)^2\right] 
            = O\left(n^{-2}(\log n)^{3d-2}\eexp{2M(8d\log n)^{\frac{k}{2}}}\right).
    \end{equation}
    
\end{corollary}
\begin{proof}
    Recall the estimator $ \widehat I_n^R(h)$ defined by~\eqref{eq:p_esitmator}. It suffices to note that
    \begin{align}
        \E\left[\left(\widehat I_n(h) - \E[h(W)]\right)^2\right] &\le 3\E\left[(\widehat I_n(h) - \widehat I_n^R(h))^2\right] + 3\E\left[\left( \widehat I_n^R(h) - \E\left[ h\circ P_R(W)\right]\right)^2\right] \notag\\
        &+ 3\E\left[\left( h\circ P_R(W) - h(W)\right)^2\right], \notag
    \end{align}
    and 
    \begin{align}
        \E\left[\left(\widehat I_n(h) - \widehat I_n^R(h)\right)^2\right]&= \E\left[\left( \frac{1}{n}\sum_{j =1}^n h\circ \Phi^{-1}(\y_j) - h\circ P_R \circ \Phi^{-1}(\y_j)\right)^2\right] \notag\\
        & \le \frac{1}{n^2} n\sum_{j =1}^n \E\left[\left(h(W)-h\circ P_R(W) \right)^2\right]\notag\\& = \E\left[\left(h(W)-h\circ P_R(W) \right)^2\right]\notag.
    \end{align}
    The desired results can be obtained by the same way as in  Theorem~\ref{them:pm-qmc}.
\end{proof}
\begin{Remark}
    When using the RQMC point set as samples, we can drop the projection operator and simply use $ \widehat  I_n(h)$ as the estimator. The RMSE rate is still the same. Due to Propositions~\ref{prop:scr1} and~\ref{prop:scr2}, a scrambled $ (t,m,d)$-net satisfies the condition in Corollary~\ref{cor:pqmc}. Owen~\cite{owen2006a} only provided the convergence rate in the sense of mean error.  He et al.~\cite{he2023} showed that this rate holds also for RMSE by using scrambled digital nets. But this is not true if using other randomization methods, such as digitally shifted nets. In contrast, we use a different framework to provide the RMSE rate for general RQMC point sets satisfying $\E\left[ D^{*}_n\left(\{\y_1,\dots,\y_n\}\right)\right] =O(n^{-1}(\log n)^{d-1})$, including both digitally shifted nets and scrambled nets.
    
\end{Remark}

\section{Importance sampling based methods}\label{sec:is}

Importance sampling methods reduce variance and speed up convergence in Monte Carlo (MC) methods by choosing a suitable importance density. However, in QMC methods, there is no intuitive way to theoretically justify the role of importance sampling as in MC methods. 

We will show theoretically how importance sampling improves QMC. We use a distribution with heavier tails than normal to achieve two things.  First, as the projection radius $ R$ goes to infinity, importance sampling preserves the convergence rate of the projection error to zero. Second, the heavy-tailed density slows down the QMC error divergence to infinity significantly. Moreover, by applying the importance sampling and replacing the low-discrepancy points with scrambled net, we obtain the convergence rate $O(n^{-3/2+\eps})$ which is better than Owen~\cite{owen2006a} and He et al.~\cite{he2023}.


We give a brief introduction to the importance sampling method. Suppose that $ \varphi$ and $ g$ are the density functions of $d$-dimensional random variables $ W$ and $ \tau$, respectively, then for any $ h$ is integrable under the measure induced by $W$, we have
\begin{align}
    \E\left[ h\left( W\right)\right] &= \int_{\R^d} h(\x) \varphi(\x)d\x \notag\\
    &= \int_{\R^d} \varphi(\x) \frac{h(\x)}{g(\x)} g(\x) d\x= \E\left[ h_{\mathrm{IS}} \left( \tau\right)\right]\notag,
\end{align}
where 
\begin{equation}\label{eq:isdef}
    h_{\mathrm{IS}} := \varphi\frac{h}{g}
\end{equation}
is the weighted function obtained after using importance sampling. 
Moreover, suppose that $\tau$ is a $d$-dimensional random variable with each component being independent of each other, i.e., the density function of $\tau = (\tau_1,\dots,\tau_d)$ has  the form
\begin{equation}
    g(\x) = \prod_{j = 1}^{d} g_j(x_j),
\end{equation}
where $ g_j$ is the marginal density function of $ \tau_j$. Let $F_j$ be the distribution function of $\tau_j$. For $ \y = (y_1,\dots,y_d)$, denote $ F^{-1}(\y)  = (F_1^{-1}(y_1),\dots,F_d^{-1}(y_d))$.

Our method is to apply the projection operator to $h_{\mathrm{IS}}$ and using the IS-based P-QMC method to obtain the quadrature rule
\begin{equation}\label{eq:is-pqmc}
    \widehat I_n^R(h_{\mathrm{IS}}) := \frac{1}{n}\sum_{j = 1}^{n} h_{\mathrm{IS}}\circ P_{R} \circ F^{-1} (\boldsymbol{y}_j) ,
\end{equation}
where $\{\boldsymbol{y}_j\}_{j =1}^{n}$ is a low-discrepancy point set. Moreover, if $ \left\{ \y_j\right\}$ is a scrambled $ (\lambda,t,m,d)$-net, then we drop the projection operator and use IS-based RQMC method to obtain the estimator
\begin{equation}\label{eq:is-rqmc}
    \widehat I_n(h_{\mathrm{IS}}) : = \frac{1}{n} \sum_{j =1}^{n} h_{\mathrm{IS}}\circ F^{-1}\left( \y_j\right).
\end{equation}

 We claim that our IS-based P-QMC and IS-based RQMC method can give the convergence rate for integrands in fast growth classes (see Definition~\ref{def:fast}), which are not ``QMC-friendly", and achieve better convergence rates of $ O\left( n^{-1+\eps}\right)$ and $ O\left( n^{-3/2+\eps}\right)$, respectively. Note that $G_e(M,B,2)$ is a larger set, so our convergence rate still holds for other growth conditions. In the following parts, we consider that $ h/g$ has the fast growth, i.e., $h/g \in G_e(M,B,2) \ \text{where}\  M <1/2$. 

\subsection{Importance sampling based P-QMC methods}
In this part, we derive the projection error caused by the projection method for the integrand after importance sampling, and then we derive the upper bound of variation in the sense of Hardy and Krause. By selecting an appropriate projection radius $ R$, we obtain the convergence rate of importance sampling based P-QMC method. The following lemma provides an estimate for the derivatives of the weighted function $ h_{\mathrm{IS}}$. 
\begin{lemma}\label{lemma:derivativeis}
    If $ h/g \in G_e\left( M,B,2\right),\ M < 1/2$ and $ \uu \subseteq 1{:}d$, then 
    \begin{equation}\label{eq:derivativeis}
         \left|\partial^{\boldsymbol{u}} h_{\mathrm{IS}}(\x)\right| \le  2^{|\boldsymbol{u}|}B|\x|^{|\boldsymbol{u}|}\eexp{-(\frac{1}{2}-M)|\x|^{2}} ,
    \end{equation}
    
\end{lemma}

\begin{proof}
    For any $ \uu \subseteq 1{:}d$, note that
    \begin{align}
        \partial^{\boldsymbol{u}}h_{\mathrm{IS}} &= \partial^{\boldsymbol{u}}\left(\varphi\frac{h}{g}\right) = \sum_{\boldsymbol{u_1}+\boldsymbol{u_2} = \boldsymbol{u}} \partial^{\boldsymbol{u_1}}\varphi\partial^{\boldsymbol{u_2}}\frac{h}{g} \notag\\
        & = e^{-\frac{|\x|^2}{2}}\sum_{\boldsymbol{u_1}+\boldsymbol{u_2} = \boldsymbol{u}} \partial^{\boldsymbol{u_2}}\frac{h}{g}\prod_{j\in \boldsymbol{u_1}}(-x_j), \notag
    \end{align}
    where $ \uu_1+\uu_2 = \uu $ means that $ \uu_1 \cup \uu_2 = \uu $ and $ \uu_1 \cap \uu_2 = \varnothing$. Taking the modulus on both sides of the above formula, we obtain
    \begin{align}
        \left|\partial^{\boldsymbol{u}} h_{\mathrm{IS}}(\x)\right| &\le e^{-\frac{|\x|^2}{2}}\sum_{\boldsymbol{u_1}+\boldsymbol{u_2} = \boldsymbol{u}} Be^{M|\x|^2}\prod_{j\in \boldsymbol{u_1}}\left| x_j\right|\notag\\
        &\le  e^{-\frac{|\x|^2}{2}}\left(\sum_{\boldsymbol{u_1}+\boldsymbol{u_2} = \boldsymbol{u}} 1 \right)Be^{M|\x|^2}|\x|^{|\boldsymbol{u}|}\label{eq:lemma4-1}\\
        &\le  e^{-\frac{|\x|^2}{2}}2^{|\boldsymbol{u}|}B\eexp{M|\x|^2}|\x|^{|\boldsymbol{u}|}\label{eq:lemma4-2}.
    \end{align}
    In~\eqref{eq:lemma4-1} we use the fact that $ |x_j| \le |\x|$, and in~\eqref{eq:lemma4-2} we use 
    $\sum_{\boldsymbol{u_1}+\boldsymbol{u_2} = \boldsymbol{u}}1 = 2^{|\uu|}.$
\end{proof}

\begin{Remark}\label{rek:deriveis}
    Note that the right hand side of~\eqref{eq:derivativeis} is monotonically decreasing with respect to $ |\x|$ when $ |\x| \ge \sqrt{\frac{|\boldsymbol{u}|}{1-2M}}$ and converge to 0 when $ |\x| \rightarrow \infty$. Therefore, there exists a constant $ C(M,B,d)$ depending only on $M,B,d$, which dominates $ \sup_{\boldsymbol{u}\subseteq 1{:}d} \left|\partial^{\boldsymbol{u}} h_{\mathrm{IS}}(\x)\right|$.
\end{Remark}

The following lemma estimates the mean squared error which dominates the projection error.
\begin{lemma}\label{lemma:isintes}
    Assume  $M <1/2$ and $R \ge 1+ 1/\sqrt{1-2M}$. For every $ h$ and $ g$ satisfying $ h/g \in  G_e\left( M,B,2\right)$ and $ \E \left|\tau\right|^2 \le A$, we have 
\begin{equation}\label{eq:isexpint}
\begin{aligned}
      \E\left[ (h_{\mathrm{IS}}\circ P_R\left( \tau\right)- h_{\mathrm{IS}}\left( \tau\right)  )^2 \right]
    \le 16AB^2d(R-1)^2 \eexp{-(1-2M)(R-1)^2}.
\end{aligned}
\end{equation}
\end{lemma}
\begin{proof}
    Denote $ H = [-R+1,R-1]^d$. Note that by the definition of projection operator $ P_R$ in~\eqref{eq:smoothprojection}, $ P_R(\x) = \x$ for any $ \x \in H$. Therefore,
    \begin{align}\label{estimateoftotal}
         \E\left[ (h_{\mathrm{IS}}\circ P_R\left( \tau\right)- h_{\mathrm{IS}}\left( \tau\right)  )^2\right] = \int_{\R^d\setminus H} \bigg|h_{\mathrm{IS}}(\x) - h_{\mathrm{IS}}(P_R(\x))\bigg|^2 g(\x) d\x,
    \end{align}
    where $g$ is the density function of $\tau$. By the Lagrange mean value theorem, we obtain
    \begin{align}
        \bigg|h_{\mathrm{IS}}(\x) - h_{\mathrm{IS}}(P_R(\x))\bigg|^2 &= \bigg|\nabla h_{\mathrm{IS}}(\boldsymbol{\xi_x}) \cdot \left( \x - P_R(\x)\right)\bigg|^2\notag\\
        &\le \bigg| \nabla h_{\mathrm{IS}}(\boldsymbol{\xi_x})\bigg|^2 \bigg|\x - P_R(\x)\bigg|^2.
    \end{align}
    where $ \boldsymbol{\xi_x}$ lies on the line segment connecting points $ \x$ and $ P_R(\x)$, and $ \nabla$  represents the gradient operator. Note that 
    \begin{equation}\label{estimateofx}
        \bigg|\x - P_R(\x)\bigg|^2 \le 2\left( |\x|^2 + |P_R(\x)|^2\right) \le 4|\x|^2.
    \end{equation}
    And if $ |\boldsymbol{u}| = 1$, then by Lemma~\ref{lemma:derivativeis} and Remark~\ref{rek:deriveis}, we obtain
    \begin{align}
        \bigg| \partial^{\boldsymbol{u}}h_{\mathrm{IS}}(\boldsymbol{\xi_x})\bigg| &\le 2B |\boldsymbol{\xi_x}| \eexp{-(\frac{1}{2}-M)|\boldsymbol{\xi_x}|^2} \notag\\
        &\le 2B(R-1)\eexp{-(\frac{1}{2}-M)(R-1)^2}\label{decreasing}.
    \end{align}
    In~\eqref{decreasing}, note that by Proposition~\ref{prop:projection}, $ |\boldsymbol{\xi_x}| \ge R-1 $ for any $ \x \in \R^d \setminus H$ and by Remark~\ref{rek:deriveis}, the function is decreasing with respect to $ |\x|$ when $ |\x|\ge 1/\sqrt{1-2M}$. Therefore, 
    \begin{equation}\label{estimateofhis}
        \bigg|\nabla h_{\mathrm{IS}}(\boldsymbol{\xi_x})\bigg| \le 2B\sqrt{d}(R-1) \eexp{-(\frac{1}{2}-M)(R-1)^2}.
    \end{equation}
    As a result,~\eqref{eq:isexpint} follows from~\eqref{estimateoftotal},~\eqref{estimateofx} and~\eqref{estimateofhis}.
\end{proof}

\begin{Remark}\label{rek:isprojectionerror}
    By Lemma~\ref{lemma:isintes}, we can drive the estimate of projection error. Note that
    \begin{align}
          \E\bigg| h_{\mathrm{IS}}\circ P_R(\tau) -  h_{\mathrm{IS}}(\tau)\bigg|  &\le \left(\E\left[\left( h_{\mathrm{IS}}\circ P_R\left( \tau\right)- h_{\mathrm{IS}}\left( \tau\right) \right)^2\right]
        \right)^{\frac{1}{2}} \notag\\
        & \le 4\sqrt{Ad}B(R-1)\eexp{-\left( \frac{1}{2}-M\right)\left( R-1\right)^2}. \notag
    \end{align}
\end{Remark}

Next, we aim to demonstrate that importance sampling can decelerate the rate at which the variation, in the sense of Hardy and Krause, diverges towards infinity. Combining the results of Lemma~\ref{lemma:derivativeis} and Remark~\ref{rek:deriveis}, we obtain the following result.
\begin{lemma}\label{lemma:ishk}
    After using importance sampling, the variation of the modified integrand $ h_{\mathrm{IS}}\circ P_R \circ F^{-1}$ in the sense of Hardy and Krause has a uniform upper bound for every $ h$ and $ g$ satisfying $ h/g \in G_e(M,B,2)$
    \begin{equation}\label{eq:ishk}
        \sup_{h/g \in G_e(M,B,2)} V_{\mathrm{HK}}(h_{\mathrm{IS}}\circ P_R \circ F^{-1}) \le C(M,B,d)2^{2d}R^{d} .
    \end{equation}
\end{lemma}

\begin{proof}
    The proof is nearly the same as in Lemma~\ref{lemma:HK}, and all we need is to note that by Lemma~\ref{lemma:derivativeis} and Remark~\ref{rek:deriveis}, every mixed derivative of $ h_{\mathrm{IS}}$ is dominated by some constant only depends on  $M,B$ and $ k$. Therefore, for any $ \uu \subseteq 1{:}d$, we can rewrite~\eqref{HK:estimate2} as 
    \begin{align}
        \int_{\left[0,1\right]^d} &\left|\partial^{\uu} h_{\mathrm{IS}}\circ P_R \circ F^{-1}\left(\boldsymbol{y}_{\boldsymbol{u}};\mathbf{1}_{-\boldsymbol{u}}\right)\right| d\boldsymbol{y} \notag\notag\\
        &\ \ \ \le C\left(M,B,d\right) \prod_{j\in \uu}\int_{-R}^Rdx_{j}\prod_{j\in \overline{\uu}}\int_{-\infty}^{\infty}g_j(x_j)d x_j\notag\\ 
        &\ \ \ \le C\left(M,B,d\right)2^dR^d\notag,
    \end{align}
    and~\eqref{eq:ishk} follows from that there are at most $ 2^d$ such terms in the definition of $ V_{\mathrm{HK}}$ (see Definition~\ref{def:HK}).
\end{proof}

Combining Remark~\ref{rek:isprojectionerror} and Lemma~\ref{lemma:ishk} and using the same method in the proof of Theorem~\ref{them:pm-qmc}, we obtain the convergence rate for importance sampling based P-QMC method. 
\begin{theorem}\label{them:isPM-QMC}
    Assume $ M < 1/2$ and $ \E \left|\tau\right|^2 \le A$. Let $ \left\{\y_1,\dots ,\y_n\right\}  $ be a low-discrepancy point set. 
    By choosing $R = \sqrt{\frac{2}{1-2M}\log n} + 1$, we obtain 
    \begin{equation}\notag
        \sup_{\substack{h/g\in G_e(M,B,2) \\ \E\left|\tau\right|^2\le A}}\bigg| \widehat I_n^R(h_{\mathrm{IS}})- \E\left[h(W)\right]\bigg| = O\left(n^{-1}(\log n)^{\frac{3d}{2}-1}\right).
    \end{equation}
    
\end{theorem}
    
\begin{Remark}
    Note that this theorem tells us that by combining importance sampling, for polynomial growth classes, the degree of $ \log n$ on the denominator of the convergence rate has been reduced from $3d/2 + k/2 -1$ to $ 3d/2-1$. For exponential growth classes, there is no exponential term on the denominator of the convergence rate. Moverover, for fast growth class, without IS, the convergence rate is $O\left( n^{-1+2M+\eps}\right)$, and by applying IS, we achieve the convergence rate $ O(n^{-1}(\log n)^{3d/2-1})$. This shows that after combining importance sampling, theoretically we obtain a faster convergence rate.
\end{Remark}

\subsection{Importance sampling based RQMC methods}
The results of the previous sections are based on the framework of the Koksma-Hlawka inequality, so they hold for any low-discrepancy point set. When replacing the QMC points with randomized QMC (RQMC) points, we obtain the importance sampling based RQMC method. This method can achieve a faster convergence rate of $O(n^{-3/2+\epsilon})$.

Owen~\cite{owen1997b,owen2008} gave a variance estimation of scrambled net for smooth functions on $ [0,1]^{d}$, and the variance is related to the infinity norm of the integrand. One thing to note is that the original integrand $ h\circ \Phi^{-1}$ does not satisfy the condition, because it has singularities. However, when we use the projection method, that is, the original integrand is composed with the projection operator, the modified integrand $ h\circ P_R\circ \Phi^{-1}$ or $ h_{\mathrm{IS}}\circ P_R \circ F^{-1}$ is smooth and we can estimate its infinity norm. The following result is from Owen~\cite{owen2008}, which consider the scrambled $ \left( \lambda,t,m,d\right)$-net (see Definition~\ref{def:tmd-net}).
\begin{lemma}\label{lemma:owen}
Let $ f(\x)$ be a smooth function. Suppose
that $ \{\boldsymbol{a}_1,\dots,\boldsymbol{a}_n\}$ is a $(\lambda, t,m, d)$-net in base $ b\ge2$ with $ n = \lambda b^m$. If $ \{\boldsymbol{y}_j\}$ is the
 scrambled version of $ \{\boldsymbol{a}_j\}$, then with $ 1\le \lambda <b$ and $ m > t+d$,
\begin{equation}
    \bvar{\frac{1}{n}\sum_{j =1}^d f(\y_j)}  \le C\frac{(\log n)^{d-1}}{n^3} \max_{\boldsymbol{u}\subseteq 1{:}d}\| \partial^{\boldsymbol{u}}f\|_{\infty}^{2},
\end{equation}
where $ C$ is a constant that depends only on $ b,\lambda,t,d $.
\end{lemma}

Recall that $ h_{\mathrm{IS}}\circ P_R \circ F^{-1} \in \mathcal{S}^d([0,1]^{d})$. Therefore, it satisfies the condition of the Lemma~\ref{lemma:owen}. Next, we give an upper bound of the infinity norm of the modified integrand $ h_{\mathrm{IS}}\circ P_R \circ F^{-1}$.

\begin{lemma}\label{lemma:isinftynorm}
    Suppose that $ M < 1/2$, then for $\uu \subseteq 1{:}d$ and $ h$ and $ g$ satisfying $ h/g \in G_e(M,B,2)$, there exists a constant $C(M,B,d)$ such that
    \begin{equation}
        \| \partial^{\boldsymbol{u}}(h_{\mathrm{IS}}\circ P_R \circ F^{-1})\|_{\infty} \le C(M,B,d)\prod_{j \in \boldsymbol{u}} \sup_{|x_j|\le R} \frac{1}{g_j(x_{j})}\label{eq:ismax} .
    \end{equation}
\end{lemma}
\begin{proof}
    Note that 
    \begin{align}
        \partial^{\boldsymbol{u}}\left( h_{\mathrm{IS}}\circ P_R \circ F^{-1}(\y)\right) &= \left( \partial^{\boldsymbol{u}}h_{\mathrm{IS}}\right)\circ P_R \circ F^{-1}(\y) \cdot \prod_{j \in  \boldsymbol{u}}\left( P_R^{\prime}(x_j)\frac{1}{g_j(x_j)}\right) \notag\\
        &\le \left( \partial^{\boldsymbol{u}}h_{\mathrm{IS}}\right)\circ P_R \circ F^{-1}(\y)\cdot \prod_{j \in \boldsymbol{u}}\left( \mathbf{1}_{\left\{ |x_j| \le R\right\}}\frac{1}{g_j(x_j)}\right)\label{eq:lemmasup},
    \end{align}
    where $ \y = (y_1,\dots,y_d)$ and $ x_j = F_{j}^{-1}(y_j)$. By Lemma~\ref{lemma:derivativeis} and Remark~\ref{rek:deriveis},~\eqref{eq:ismax} follows from~\eqref{eq:lemmasup}.
\end{proof}

Furthermore, we impose some specific restrictions on the importance density. Assumme $ \tau$ is square integrable. We consider separately the cases where the importance density has polynomial and exponential growth.   Define the following two sets: 
    \begin{align}
        V_p(k) : = \bigg\{ (h,\tau) : \frac{h}{g} \in G_e(M,B,2),  \ \max_{1\le j \le d} \frac{1}{g_j (x)} \le  M_1|x|^{k}+B_1\ 
        \text{and}\ \E \left|\tau\right|^2\le A\bigg\}\notag,
    \end{align}
    \begin{align}
        V_e(k): = \bigg\{ (h,\tau) : \frac{h}{g} \in G_e(M,B,2),  \ \max_{1\le j \le d} \frac{1}{g_j (x)} \le B_2e^{M_2|x|^{k}}\ 
        \text{and}\ \E\left|\tau\right|^2\le A\bigg\}\notag.
    \end{align}
As before, we can balance the errors by choosing an appropriate projection radius $ R$ and drive the convergence rate for the IS-based RMQC method.    
\begin{theorem}\label{them:isPM-RQMC}
    Let $ \{\y_1,\dots,\y_n\}$ be a scrambled $(\lambda,t,m,d)$-net in base $ b$ with $ n = \lambda b^m$, and suppose $ M < 1/2$. We have
    \begin{equation}\label{eq:ISRQMCresult}
        \sup_{(h,\tau) \in V_p(k)}\E\left[ (\widehat I_n(h_{\mathrm{IS}}) - \E\left[h(W)\right])^2\right] = O\left(n^{-3}(\log n)^{(k+1)d}\right),
    \end{equation}  
    and 
    \begin{equation}\notag
        \sup_{(h,\tau) \in V_e(k)}\E\left[( \widehat I_n(h_{\mathrm{IS}}) - \E\left[h(W)\right])^2\right] = O\left(n^{-3}(\log n)^{d}\eexp{2dM_2\left( \frac{3\log n}{1-2M}\right)^{\frac{k}{2}}}\right).
    \end{equation} 
\end{theorem}
\begin{proof}
    It suffices to prove~\eqref{eq:ISRQMCresult}. Let $ (h,\tau) \in V_p(k)$ and let $ \{\y_1,\dots,\y_n\}$ be a scrambled $ \left( 
    \lambda, t,m,d\right)$-net. We have
    \begin{align}\label{eq:decomposition}
        \E\bigg[ & \left(\widehat I_n(h_{\mathrm{IS}}) - \E\left[h(W)\right]\right)^2\bigg]\le 3\E\bigg[\left(\widehat I_n(h_{\mathrm{IS}}) - \widehat I_n^R(h_{\mathrm{IS}})\right)^2\bigg]\\
        &+ 3\E\bigg[\left(\widehat I_n^R(h_{\mathrm{IS}})-\E\left[h_{\mathrm{IS}}\circ P_R(\tau) \right] \right)^2\bigg] +3\left(\E\left[h_{\mathrm{IS}}\circ P_R(\tau)\right] - \E\left[h(W)\right] \right)^2\notag.
    \end{align}
    For the first term of the right hand side of~\eqref{eq:decomposition}, 
    \begin{align}
         \E\bigg[\left( \widehat I_n(h_{\mathrm{IS}})- \widehat I_n^R(h_{\mathrm{IS}})\right)^2\bigg] &= \E\left[\left( \frac{1}{n} \sum_{j =1}^n\left( h_{\mathrm{IS}}\circ P_R\circ F^{-1}(\y_j) - h_{\mathrm{IS}}\circ F^{-1}(\y_j)\right)\right)^2\right] \notag\\
         &\le \frac{1}{n^2} n \sum_{j =1}^n \E\left[\left( h_{\mathrm{IS}}\circ P_R\circ F^{-1}(\y_j) - h_{\mathrm{IS}}\circ F^{-1}(\y_j) \right)^2\right] \label{eq:term1-1}\\
         & = \E \left[\left( h_{\mathrm{IS}}\circ P_R(\tau) - h_{\mathrm{IS}}(\tau)\right)^2\right], \label{eq:term1-2}
    \end{align}
    where in~\eqref{eq:term1-1}, we use the Cauchy–Schwarz inequality, and~\eqref{eq:term1-2} follows from that every $ \y_j$ has the uniform distribution on $ [0,1]^d$ (see Proposition~\ref{prop:scr2}).

    Note that the second term of the right hand side of~\eqref{eq:decomposition} is $ 3\var{\widehat I_n^R(h_{\mathrm{IS}})}$ and the third term is bounded by $ 3\E \left[( h_{\mathrm{IS}}\circ P_R(\tau) - h_{\mathrm{IS}}(\tau))^2\right]$. Therefore, we obtain
    \begin{align}
        \E\bigg[\left( \widehat I_n(h_{\mathrm{IS}}) - \E\left[h(W)\right]\right)^2\bigg] \le 3\var{\widehat I_n^R(h_{\mathrm{IS}})} + 6\E \left[\big( h_{\mathrm{IS}}\circ P_R(\tau) - h_{\mathrm{IS}}(\tau)\big)^2\right].\label{eq:term2}
    \end{align}
    Using the results in Lemma~\ref{lemma:owen} and Lemma~\ref{lemma:isinftynorm}, when $ n $ is large enough so that $ m> t+d$, the first term of the right hand side of~\eqref{eq:term2} has the following upper bounds.
    \begin{align}
        3\var{\widehat I_n^R(h_{\mathrm{IS}})} &\le C \max_{\boldsymbol{u}\subseteq 1{:}d}\| \partial^{\boldsymbol{u}}(h_{\mathrm{IS}}\circ P_R \circ F^{-1})\|_{\infty}^{2}\frac{(\log n)^{d-1}}{n^3}\notag\\
        &\le C\prod_{j = 1}^{d} \sup_{|x_j|\le R} \frac{1}{g_j^2(x_{j})} \frac{(\log n)^{d-1}}{n^3}\label{eq:the1-3}\\
        &\le C\left( M_1 R^{k}+B_1\right)^{2d}\frac{(\log n)^{d-1}}{n^3}\label{eq:the1-4},
    \end{align}
    where $ C$ is a constant only depends on $  M,B,M_1,t,b,d$. The inequality~\eqref{eq:the1-3} follows from Lemma~\ref{lemma:isinftynorm}, and~\eqref{eq:the1-4} follows from $ \max_{1\le j \le d} \frac{1}{g_j} \le M_1|x|^k +B_1$.
    
    By Lemma~\ref{lemma:isintes},
    \begin{align}
        6\E \left[\left( h_{\mathrm{IS}}\circ P_R(\tau) - h_{\mathrm{IS}}(\tau)\right)^2\right] \le 96AB^2d(R-1)^2\eexp{-(1-2M)(R-1)^2}.\label{eq:term3}
    \end{align}
    To balance two terms~\eqref{eq:the1-4} and~\eqref{eq:term3}, we choose 
    \begin{equation}\notag
         R = \sqrt{\frac{3}{1-2M}\log n} + 1.
    \end{equation}
    Therefore,~\eqref{eq:the1-4} achieves $ O\left( n^{-3}(\log n)^{(k+1)d-1}\right)$, and~\eqref{eq:term3} achieves $ O\left( n^{-3}\log n\right)$. This proves the desired result.
\end{proof}

\begin{Remark}
    It is important to note that without importance sampling, the RQMC method cannot achieve a higher convergence rate of $ O(n^{-3/2+\eps})$, due to the infinity norm of the derivatives of the integrand cannot be controlled by a constant and the right hand side of~\eqref{eq:ismax} is $ O(\eexp{dR^2})$, thus we can not find a projection radius $ R$ to make the projection error and QMC error both converge at a rate of $ O(n^{-3/2+\eps}).$
\end{Remark}

The conditions for Theorems~\ref{them:isPM-QMC} and~\ref{them:isPM-RQMC} are about $ h$ and $ g$. If we fix an appropriate importance density, the conditions of the theorems can be directly restricted to $ h$. The following section will provide a detailed discussion on this perspective. Note that, the next section uses the $ t$-distribution as an example of an IS proposal. However, our framework is highly versatile, and many distributions with heavier tails than the normal distribution may meet the conditions of Theorem~\ref{them:isPM-RQMC} and be used as an IS proposal.

\subsection{The choice of the IS density}\label{subsec:ISchoice}
We use a heavy-tailed distribution as the IS proposal to accelerate the convergence rate. We introduce a class of importance densities that satisfy the conditions of the above theorems. Take the $ t$-distribution as the proposal of importance sampling. The density function of $\tau$ is 
\begin{equation}\label{eq:ISt}
    g(\x) = \prod_{j = 1}^d g_j(x_j) =  \prod_{j =1}^d \frac{\Gamma(\frac{\nu_j+1}{2})}{\sqrt{\nu_j\pi}\Gamma(\frac{\nu_j}{2})}\left( 1+\frac{x_j^2}{\nu_j}\right)^{-(\nu_j+1)/2}.
\end{equation}
Each component of $ \tau$ is a $ t$-distribution with parameter $ \nu_j$. Let $ \nu = \max_j \nu_j$. We can easily verify that there exist $ M(\nu)>0$ and $ B(\nu)>0$, such that
\begin{equation}\notag
    \max_{1\le j\le d} \frac{1}{g_j(x)} = \max_{1\le j \le d} \frac{\sqrt{\nu_j\pi}\Gamma(\frac{\nu_j}{2})}{\Gamma(\frac{\nu_j+1}{2})}\left( 1+\frac{x^2}{\nu_j}\right)^{(\nu_j+1)/2} \le M(\nu)|x|^{\nu+1} + B(\nu) ,
\end{equation}
and 
\begin{equation}\notag
    \frac{1}{g} \in  G_p(M(\nu),B(\nu),\nu+1).
\end{equation}
Therefore, for any function $ h\in G_e(M,B,2)$ with $ M < 1/2$, using the (3) of Theorem~\ref{them:operation} in Appendix, we have that for any $ \eps > 0 $, there exist $ B(\eps)>0$, such that
\begin{equation}\notag
    h/g \in  G_e(M+\eps,B(\eps),2).
\end{equation}
Note that $ \E \left|\tau\right|^2< \infty$ when $ \nu_j \ge 3$ for all $ 1\le j \le d$. By choosing the $ \eps $ small enough so that $M+\eps<1/2$ and $ \nu_j \ge 3$, we obtain 
$$ (h,\tau) \in V_p(\nu+1).$$ 
Therefore, it satisfies the conditions in Theorems~\ref{them:isPM-QMC} and~\ref{them:isPM-RQMC}, and we can drive the following theorem.
\begin{theorem}\label{them:fixIS}
    Assume $ g$ satisfies~\eqref{eq:ISt} with $ \nu_j \ge 3$ for all $ 1 \le j \le d$ and $ M < 1/2$. Let $ \nu = \max_j \nu_j$. \\
    (\uppercase\expandafter{\romannumeral1}) If $ \{\y_1,\dots,\y_n\}$ is a low-discrepancy point set, then by choosing $ R = \sqrt{\frac{2\log n}{1-2M}}+1$, we have
    \begin{equation}\notag
        \sup_{h\in G_e(M,B,2)} \bigg|\widehat I_n^R(h_{\mathrm{IS}}) - \E\left[ h(W)\right]\bigg| = O(n^{-1}(\log n)^{\frac{3d}{2}-1}).
    \end{equation}
    (\uppercase\expandafter{\romannumeral2}) If $ \{\y_1,\dots,\y_n\}$ is a scrambled $(\lambda,t,m,d)$-net in base $ b$ with $ n = \lambda b^m$, then
    \begin{equation}\notag
        \sup_{h\in G_e(M,B,2)} \E\left[\left( \widehat I_n(h_{\mathrm{IS}}) - \E\left[ h(W)\right]\right)^2\right] = O(n^{-3}(\log n)^{(\nu+2)d}).
    \end{equation}
\end{theorem}

\begin{Remark}
    If we choose such an importance density with polynomial growth, then IS-based P-QMC or IS-based RQMC method can handle the case where the function $ h \in G_e(M,B,2), M < 1/2$, which is not ``QMC-friendly", and achieve the convergence rate of $ O(n^{-1+\eps})$ and $O(n^{-3/2+\eps}) $, respectively. However, by the results of Owen~\cite{owen2006a}, the convergence rate of QMC is  $ O(n^{-1+2M+\eps})$ without IS. This shows that, IS does accelerate the convergence rate in QMC.
\end{Remark}

\section{Numerical results}\label{sec:numerical}

In numerical experiments, we use the Sobol' sequence, whose first $ 2^m$ points constitute a $ \left( t,m,d\right)$-net in base $ 2$, and use the scrambled Sobol' sequence in RQMC methods.

We focus on comparing the convergence results of RQMC and IS-based RQMC on the fast growth class. We use the test function 
\begin{equation}\notag
    h(\x) = C\eexp{M\left| \x\right|^2},
\end{equation}
where $ 0< M < 1/2$ has an impact on the boundary growth and we take $C = \left( 1-2M\right)^{d/2}$ to ensure $ \E[h(W)] = 1$ for all $d\ge 1$. It is easy to verify that this function does not satisfy the ``QMC-friendly'' condition, and so using RQMC directly without IS will result in a convergence rate of $ O\left( n^{-1+2M+\eps}\right)$. We take the $ t$-distribution with $ \nu = 3$ as the proposal of importance sampling (see details in Section~\ref{subsec:ISchoice}). In comparison, the convergence rate of IS-based RQMC is imporved to $ O\left( n^{-3/2+\eps}\right)$. RMSEs in the following numerical results are computed based on 100 independent repetitions.

\begin{figure}[htbp]
  \centering 
\includegraphics[height=8cm,keepaspectratio]{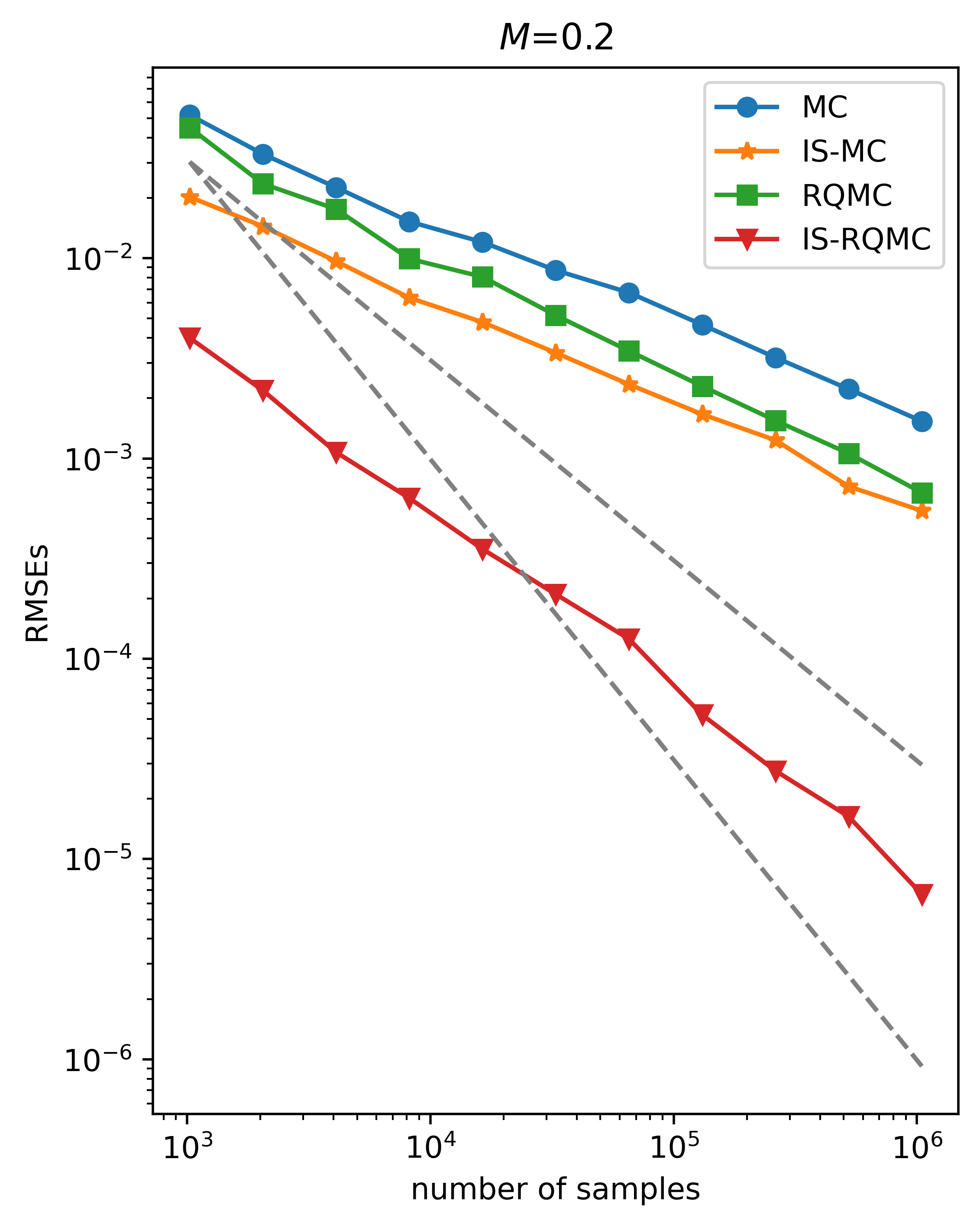}
 \hfill  
\includegraphics[height=8cm,keepaspectratio]{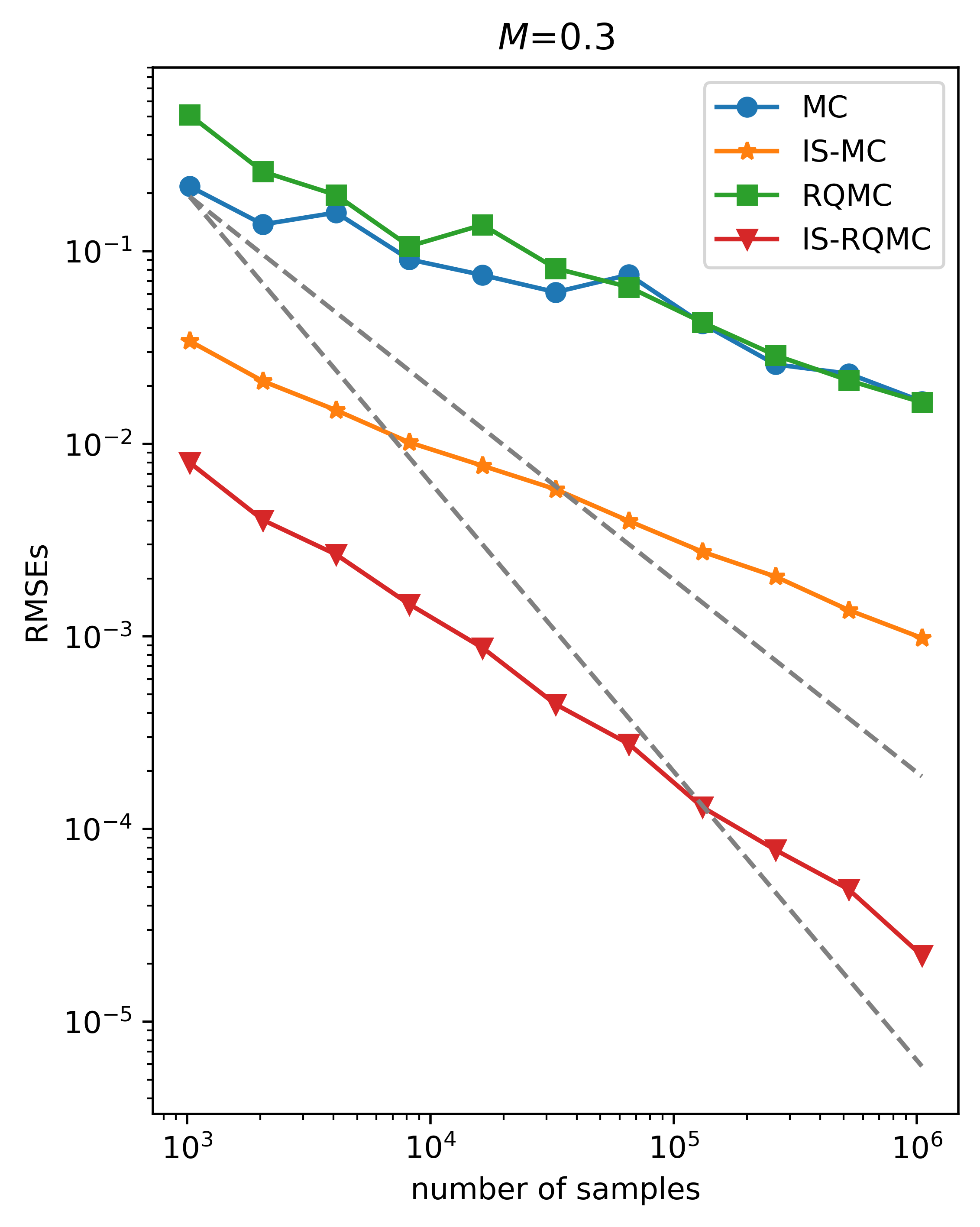}
  \caption{RMSEs for the test function $ h$ with $ d = 5$ and $ \nu = 3$. The RMSEs are computed based on $ 100$ repetitions. The slopes of the gray dashed lines are $ -1$ and $ -3/2$.}
  \label{fig:fd=5}
\end{figure}

For $ d = 5$, we compare the RMSEs of MC, IS-based MC, RQMC and IS-based RQMC as the sample size increases in Figure~\ref{fig:fd=5}. The figure shows that IS-based RQMC has much smaller RMSE and better convergence rate than MC, IS-based MC and RQMC. 

Note that when $ M = 0.2$, the test function has fast growth and $ M < 1/2$. Without importance sampling, RQMC converges at a rate $ O\left( n^{-0.6+\eps}\right)$  and the convergence rate of MC is $ O(n^{-1/2})$, while if we use importance sampling, the convergence rate will reach $ O\left( n^{-3/2+\eps}\right)$. This numerical result agrees with our theoretical results.

When $ M= 0.3$, the variance of $ h(W)$ is infinite, so MC does not converge and RQMC has a bad performance. However, IS-based RQMC still has a good convergence rate when the sample is large enough. Note that after IS, the variance of $ h_{\mathrm{IS}}$ is finite, so the IS-based MC also converges.

\begin{figure}[htbp]
  \centering 
\includegraphics[height=8cm,keepaspectratio]{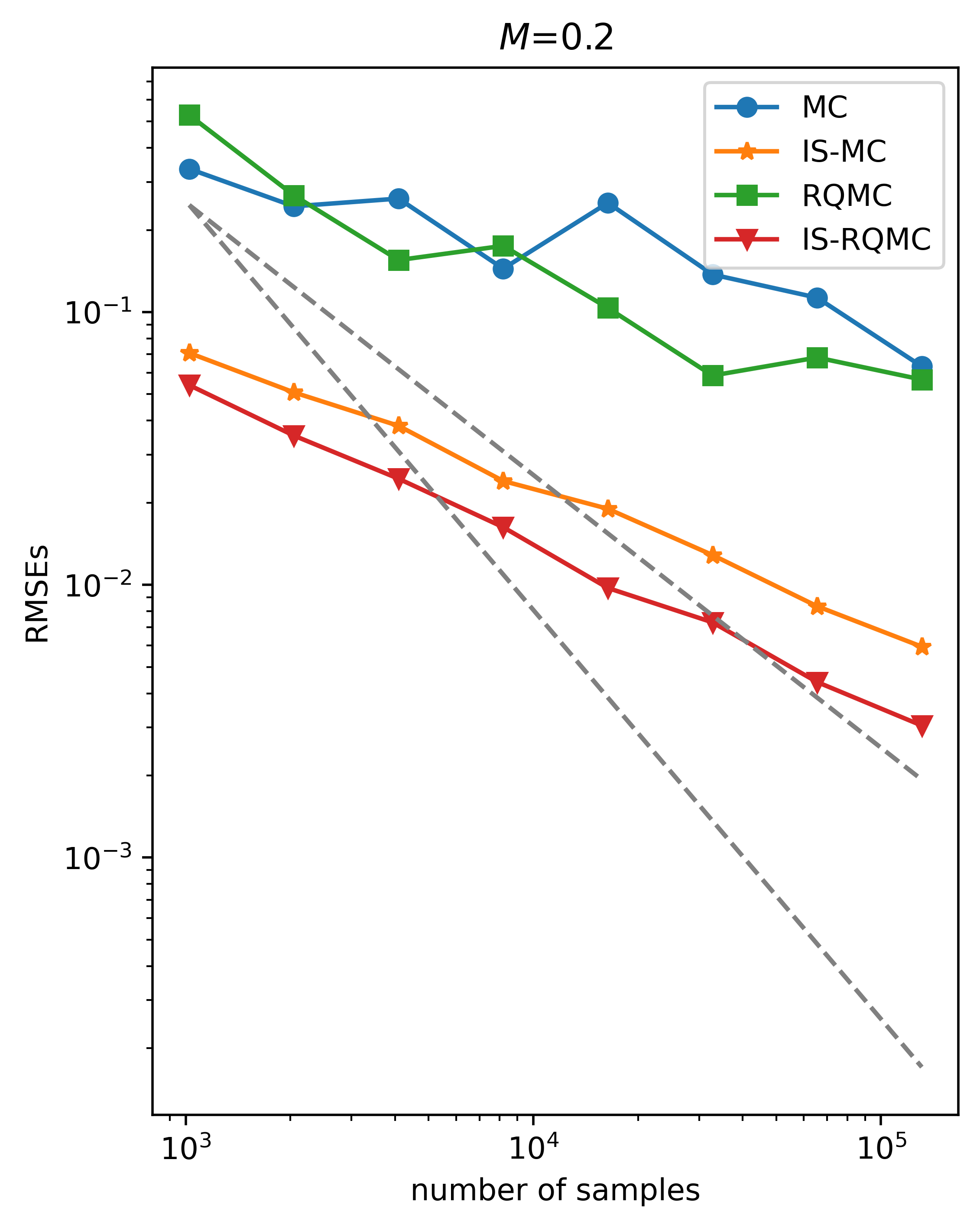}
 \hfill  
\includegraphics[height=8cm,keepaspectratio]{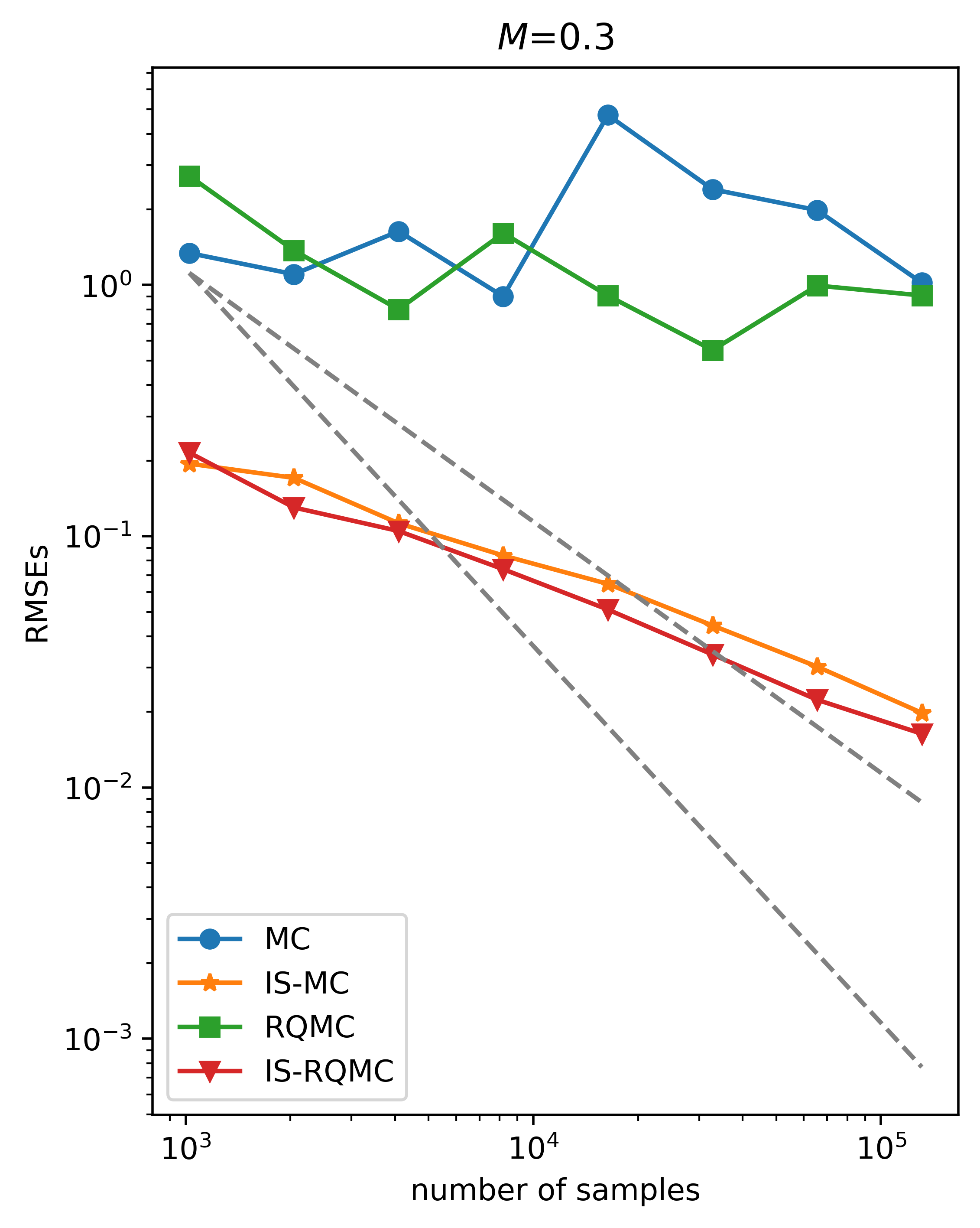}
  \caption{RMSEs for the test function $ h$ with $ d = 30$ and $ \nu = 3$. The RMSEs are computed based on $ 100$ repetitions. The slopes of the gray dashed lines are $ -1$ and $ -3/2$.}
  \label{fig:fd=30}
\end{figure}
When we increase $ d$ to 30, the  convergence rate of IS-based RQMC in then sense of RMSE is $ O(n^{-3/2}(\log n)^{75})$. Therefore, with a limited number of samples, the slope of the IS-based RQMC curve may not reach $ -3/2$. However, the numerical results in Figure~\ref{fig:fd=30} indicate that IS-based RQMC is still the most effective method.

Note that we only provided the $ t$-distribution as an example. In fact, any distribution that has heavier tails than the normal distribution, such as some distributions in the exponential family, may satisfy our conditions to be used as a proposal for IS and achieve a higher convergence rate of $ O(n^{-3/2+\eps})$. We find that reducing variance is not the only criterion for IS in QMC. The picture on the left side of Figure~\ref{fig:fd=5} shows that the effect of IS-based RQMC is significantly improved when the variance is not reduced too much. How to choose an IS proposal that yields better results for a given integrand, that will be a question for our future research.

\section{Conclusion}\label{sec:conclusion}
Using the projection based quasi-Monte Carlo method on various growth classes, we attained more refined results than Owen~\cite{owen2006a}'s. Our framework dictates that the convergence rate is contingent upon the  growth condition of the integrand. For integrand with fast growth, both QMC and MC methods manifest inferior performance. Nevertheless, by using importance sampling with a heavy-tailed proposal, we achieved better convergence rates. This assertion is corroborated by our numerical experiments. In this paper, we do not focus on reducing the Monte Carlo variance through importance sampling. It is desirable to develop a good importance sampling to reduce Monte Carlo variance while retaining the faster convergence rate of QMC. This is left for future work.

\section*{Appendix}
In many problems, the integrand that we consider, such as the loss function in deep learning, has a very complex form (it is obtained by some basic operations on different functions). We refer to Huré et al.~\cite{hure2020} and Beck et al.~\cite{beck2021} for details. For such a complex integrand, it is not easy to verify directly which growth class it belongs to, but the growth conditions of the functions that compose the integrand are easy to verify. Therefore, we need to discuss what kind of growth class we get after applying some operations on the basic growth classes. We use the following notation.

\begin{definition}
    If $S_1$ and $S_2$ are two function classes, then for $\otimes = +,\times,\circ$, we define
    \begin{equation}\notag
        S_1 \otimes S_2 := \left\{h_1\otimes h_2:\ h_1\in S_1,\ h_2 \in S_2\right\}.
    \end{equation}
    and we define the corresponding scalar operations for the case that $ S_1$ is a constant.
\end{definition}

The growth classes that we defined have properties similar to those of a linear space, except that the corresponding parameters may change. From Theorem~\ref{them:pm-qmc}, we know that for polynomial case, the convergence rate is determined by $ k$, and for exponential case, the convergence rate is determined by $ M$ and $ k$. Therefore, in the following calculations, we will focus on the changes of these coefficients, and the specific forms of other coefficients are irrelevant for our analysis.
\begin{theorem}\label{them:operation}
    Assume $ M,M_1,M_2,B,B_1,B_2,k,k_1,k_2$ are all positive. 
    \begin{itemize}
        \item[(1)] Scalar multiplication. 
        
        Assume $ c \ne 0$ is a constant. For the polynomial case,
        \begin{equation}
            c \times G_p(M,B,k) =  G_p(|c|M,|c|B,k).\notag
        \end{equation}
        For the exponential case, 
        \begin{equation}
            c \times G_e(M,B,k) =  G_e(M,|c|B,k).\notag
        \end{equation}
        
        \item[(2)] Addition.
        
        The first case is the addition of two polynomial growth classes. There exist $ M_3>0 $ and $ B_3>0$, such that
        \begin{equation}
            G_p(M_1,B_1,k_1) + G_p(M_2,B_2,k_2) \subseteq  G_p(M_3,B_3, \max\{k_1,k_2\} )\notag.
        \end{equation}
        The second case is the addition of two exponential growth classes. If $ k_1 > k_2$, then there exists $ B_3>0$, such that
        \begin{equation}
             G_e(M_1,B_1,k_1) +  G_e(M_2,B_2,k_2) \subseteq  G_e(M_1,B_3, k_1 )\notag.
        \end{equation}
        If $ k_1 = k_2=k$, then 
        \begin{equation}
            G_e(M_1,B_1,k) + G_e(M_2,B_2,k) \subseteq G_e(\max\{M_1, M_2\},B_1+B_2, k )\notag.
        \end{equation}
        The third case is the addition of an exponential growth class and a polynomial growth class. There exists $ B_3>0$, such that
        \begin{equation}
            G_e(M_1,B_1,k_1) + G_p(M_2,B_2,k_2) \subseteq G_e(M_1 ,B_3, k_1 )\notag.
        \end{equation}

        \item[(3)]\label{item:mul} Multiplication.

        The first case is the multiplication of two polynomial growth classes. There exist $ M_3>0 $ and $ B_3>0$, such that
        \begin{equation}
            G_p(M_1,B_1,k_1) \times G_p(M_2,B_2,k_2) \subseteq G_p(M_3,B_3, k_1 + k_2)\notag.
        \end{equation}
        The second case is the multiplication of two exponential growth classes. If $ k_1 > k_2$, then for any $ \eps > 0$, there exists $ B(\eps)>0$, such that
        \begin{equation}
            G_e(M_1,B_1,k_1) \times G_e(M_2,B_2,k_2) \subseteq G_e(M_1+\eps,B(\eps), k_1 )\notag.
        \end{equation}
        If $ k_1 = k_2$, then there exists $ B_3>0$, such that
        \begin{equation}
            G_e(M_1,B_1,k_1) \times G_e(M_2,B_2,k_2) \subseteq G_e(M_1 + M_2,B_3, k_1 )\notag.
        \end{equation}
        The third case is the multiplication of an exponential growth class and a polynomial growth class. For any $ \eps >0$, there exists $ B(\eps)>0$, such that
        \begin{equation}
             G_e(M_1,B_1,k_1) \times  G_p(M_2,B_2,k_2) \subseteq G_e(M_1 +\eps,B(\eps), k_1 )\notag.
        \end{equation}
        
    \end{itemize}
\end{theorem}

\begin{proof}
    The results of scalar multiplication are easy to verify. We briefly prove the other results.   For any $$f \in  G_p(M_1,B_1,k_1),\ g\in  G_p(M_2,B_2,k_2),$$ it is easy to verify that the result holds by choosing appropriate $(M_3,B_3)$ for addition $ + $. For multiplication $ \times$, it suffices to note that for fixed $\uu \subseteq 1{:}d$,
    \begin{equation}\notag
        |\partial^{\uu}(f\times g)|  = \left|\sum_{\uu_1+\uu_2 = \uu} \partial^{\uu_1}f\partial^{\uu_2}g\right| \le \sum_{\uu_1+\uu_2 = \uu} \left|\partial^{\uu_1}f \partial^{\uu_2}g\right| .
    \end{equation}
     As for exponential case, note that for any $ \eps > 0 $, there exists constant $ B$ which depends on $ \eps$, such that 
     \begin{equation}\notag
         |\x|^{k_2}e^{M|\x|^{k_1}} \le Be^{(M+\eps)|\x|^{k_1}}.
     \end{equation} 
     It is easy to verify the conclusion of the theorem through these calculations.
\end{proof}

For investigating composition operations, we need additional symbols. Let $\Z_d$ be the $n$-fold index, i.e.,
\begin{equation}\notag
    \Z_d : = \left\{\boldsymbol{\alpha} = \left(\alpha_1,\cdots,\alpha_d\right):\alpha_j \in \N , \ j = 1,\cdots,d\right\} ,
\end{equation}
The module of $\boldsymbol{\alpha} \in \Z_d$ is defined by 
\begin{equation}\notag
    |\boldsymbol{\alpha}|:=\alpha_1+\cdots+\alpha_d,
\end{equation} 
and for $ \boldsymbol{\beta} \in \Z_d$, we say $ \boldsymbol{\alpha} \prec \boldsymbol{\beta }$ if $ \alpha_j \le \beta_j$ holds for every $ 1\le j\le d$.
We take
\begin{equation}\notag
    \boldsymbol{\alpha}! = \prod_{j = 1}^d\alpha_j!.
\end{equation}
The derivative notation to be used is then 
\begin{equation}\notag
    D^{\boldsymbol{\alpha}}h(\boldsymbol{x}) :=  \frac{\partial^{|\boldsymbol{\alpha}|} }{\partial x_1^{\alpha_1} \ldots \partial x_d^{\alpha_d}}h\left(\boldsymbol{x}\right).
\end{equation}

\begin{definition} 
For $ M>0,B>0, k>0$ and $ s,l \in \N_{+}$, define
    \begin{equation}\notag
        \widehat G_p^{(l,s)}(M,B,k):= \left\{h\in C^d(\R^s\rightarrow\R^l):\sup_{|\boldsymbol{\alpha}|\le d} \left|D^{\boldsymbol{\alpha}}h(\boldsymbol{x})\right| \le M|\x|^k+B\right\} ,
    \end{equation}
    and 
    \begin{equation}\notag
        \widehat G_e^{(l,s)}(M,B,k):= \left\{h\in C^d(\R^s\rightarrow\R^l):\sup_{|\boldsymbol{\alpha}|\le d} \left|D^{\boldsymbol{\alpha}}h(\boldsymbol{x})\right| \le Be^{M|\x|^k}\right\} .
    \end{equation}
\end{definition}

The function class defined here is different from Definition~\ref{def:growth class}, since we need to take the partial derivatives of the composited function with respect to its arguments several times.

\begin{theorem}\label{them:composition}

       Assume integers $l,s,r \le d $. For the composition of two polynomial growth classes, there exist $ M_3>0 $ and $ B_3>0$, such that
        \begin{equation}
            \widehat G_p^{(l,s)}(M_1,B_1,k_1) \circ \widehat G_p^{(s,r)}(M_2,B_2,k_2) \subseteq \widehat G_p^{(l,r)}(M_3,B_3, k_1k_2)\notag.
        \end{equation}
        For an exponential growth class compositing a polynomial growth class, there exist $ M_3>0 $ and $ B_3>0$, such that
        \begin{equation}
            \widehat G_e^{(l,s)}(M_1,B_1,k_1) \circ \widehat G_p^{(s,r)}(M_2,B_2,k_2) \subseteq \widehat G_e^{(l,r)}(M_3,B_3, k_1k_2 )\notag.
        \end{equation}
        For a polynomial growth class compositing an exponential growth class, there exist $ M_3>0 $ and $ B_3>0$, such that
        \begin{equation}
            \widehat G_p^{(l,s)}(M_1,B_1,k_1) \circ \widehat G_e^{(s,r)}(M_2,B_2,k_2) \subseteq \widehat G_e^{(l,r)}(M_3,B_3, k_2 )\notag.
        \end{equation}
\end{theorem}

\begin{proof}
    It suffices to prove the case $ l = s = 1$ and $ r = d$. For any $$f \in  \widehat G_p^{(1,1)}(M_1,B_1,k_1),\ g\in  \widehat G_p^{(1,d)}(M_2,B_2,k_2), h\in \widehat G_e^{(1,d)}(M_2,B_2,d_2).$$ Note that
    \begin{equation}\notag
        |f\circ g| \le M_1|g(\x)|^{k_1} + B_1 \le M_1\left(M_2|\x|^{k_2}+B_2\right)^{k_1} + B_1 \le M_3|\x|^{k_1k_2} + B_3 .
    \end{equation}
    Similarly,
    \begin{align}
        |f\circ h|\le M_1|h(\x)|^{k_1} + B_1 &\le M_1\left(B_2\eexp{M_2|\x|^{k_2}}\right)^{k_1} + B_1 \notag\\
        &\le M_1B_2^{k_1}\eexp{k_1M_2|\x|^{k_2}} + B_1\notag\\
        &\le \left( M_1B_2^{k_1} + B_1\right)\eexp{k_1M_2|\x|^{k_2}}\notag .
    \end{align}

    Note that for any fixed $\boldsymbol{\alpha} \in \Z_d$, ever term of $D^{\boldsymbol{\alpha}}(f\circ g)$ can be written as the combination of 
    \begin{equation}\notag
    (D^{\boldsymbol{\lambda}}f)\circ g,\ f,\ D^{\boldsymbol{\ell_j}}g,\ g,\ \ \ |\boldsymbol{\lambda}|,\ |\boldsymbol{\ell_j}| \le |\boldsymbol{\alpha}|
    \end{equation}
    under the operators $ + $, $ \times$ and scalar multiplication. More precisely, by multivariate Faa di Bruno formula (see Theorem 2.1 in~\cite{cons:1996}), we have 
    \begin{equation}\label{eq:mulderivative}
        D^{\boldsymbol{\alpha}}\left( f\circ g\right)=\sum_{1 \leq|\boldsymbol{\lambda}| \leq |\boldsymbol{\alpha}|} (D^{\boldsymbol{\lambda}}f)\circ g \sum_{s=1}^{|\boldsymbol{\alpha}|} \sum_{p_s(\boldsymbol{\alpha}, \boldsymbol{\lambda})}(\boldsymbol{\alpha} !) \prod_{j=1}^s \frac{\left[D^{\boldsymbol{\ell}_j}g\right]^{\mathbf{q}_j}}{\left(\mathbf{q}_{j} !\right)\left[\boldsymbol{\ell}_{j} !\right]^{|\mathbf{q}_j|}},
    \end{equation}
    where 
    \begin{equation}
    \begin{aligned}
    p_s(\boldsymbol{\alpha}, \boldsymbol{\lambda})=\big\{&\left(\mathbf{q}_1, \ldots, \mathbf{q}_s ; \boldsymbol{\ell}_1, \ldots, \boldsymbol{\ell}_s\right):\left|\mathbf{q}_i\right|>0,\\
    &0 \prec \boldsymbol{\ell}_1 \prec \cdots \prec \boldsymbol{\ell}_s, \sum_{i=1}^s \mathbf{q}_i=\boldsymbol{\lambda} \text { and } \sum_{i=1}^s\left|\mathbf{q}_i\right| \boldsymbol{\ell}_i=\boldsymbol{\alpha}\big\}.\notag
    \end{aligned}
    \end{equation} 
    The functions in~\eqref{eq:mulderivative} are all in the corresponding growth classes. Using the same method in Theorem~\ref{them:operation}, one can prove the desired results.    
\end{proof}

\end{document}